\documentclass[12pt,a4paper,oldfontcommands]{amsart}
\usepackage[utf8]{inputenc}
\usepackage[T1]{fontenc}
\usepackage[foot]{amsaddr}
\usepackage{microtype}
\usepackage[dvips]{graphicx}
\usepackage{xcolor}
\usepackage{times}
\usepackage{amsmath,amsfonts,amssymb,amsthm}
\usepackage{bm}
\usepackage{xr}
\usepackage{tikz}
\usepackage{pgffor}
\usepackage{subfigure}
\usepackage{mathtools}
\usepackage[
breaklinks=true,colorlinks=true,
linkcolor=black,urlcolor=black,citecolor=black,
bookmarks=true,bookmarksopenlevel=2]{hyperref}

\usepackage{geometry}
\newcommand{\tr}[1]{\textcolor{black}{#1}}
\newcommand{\I}{0,T}
\newcommand{\ito}{I\times\Omega}
\newcommand{\imto}{I_m\times\Omega}
\newcommand{\ovl}{\overline{l}}

\newcommand{\ovlth}{\overline{l}_{\tau h}}

\newcommand{\ovls}{\overline{l}_\sigma}

\newcommand{\ovv}{\overline{\varphi}}

\newcommand{\ovd}{\overline{d}}

\newcommand{\ovz}{\overline{z}}

\newcommand{\ovp}{\overline{p}}

\newcommand{\Hminuseins}{H^{-1}(\Omega)}
\newcommand{\Heins}{H^1(\Omega)}
\newcommand{\Hzwei}{H^2(\Omega)}
\newcommand{\Hs}{H^s(\Omega)}

\newcommand{\LL}[2]{L^{#1}(\I;L^{#2}(\Omega))}
\newcommand{\Heinsnull}{H^1_0(\Omega)}
\newcommand{\Lzwo}{L^2(\Omega)}

\newcommand{\Lzweizwei}{L^2(0,T;\Lzwo)}
\newcommand{\Leinszwei}{L^1(0,T;\Lzwo)}
\newcommand{\LzweiHeins}{L^2(0,T;H^1(\Omega))}
\newcommand{\LzweiHminuseins}{L^2(0,T;H^{-1}(\Omega))}

\newcommand{\Heinszwei}{H^1(\I;\Lzwo)}
\newcommand{\Heinsnullzwei}{H^1_{\{0\}}(\I;\Lzwo)}

\newcommand{\Linfzwo}{L^\infty(\I;\Lzwo)}

\newcommand{\Xtau}{X^0_\tau}
\newcommand{\Vtau}{V^0_\tau}
\newcommand{\B}{\mathfrak{B}}
\newcommand{\bb}{\mathfrak{b}}
\newcommand{\vt}{\varphi_\tau}
\newcommand{\vtm}{\varphi_{\tau,m}}

\newcommand{\dt}{d_\tau}
\newcommand{\dtm}{d_{\tau,m}}
\newcommand{\evt}{e_\tau^\varphi}
\newcommand{\edt}{e_\tau^d}
\newcommand{\etavt}{\eta_\tau^\varphi}
\newcommand{\xivt}{\xi_\tau^\varphi}
\newcommand{\etadt}{\eta_\tau^d}

\newcommand{\xidt}{\xi_\tau^d}
\newcommand{\zt}{z_\tau}
\newcommand{\pt}{p_\tau}
\newcommand{\Vh}{V_h^1}
\newcommand{\Xh}{X^1_h}
\newcommand{\Vth}{V^{0,1}_{\tau h}}
\newcommand{\V}[2]{V_{\tau h}^{#1,#2}}
\newcommand{\Xth}{X^{0,1}_{\tau h}}
\newcommand{\vth}{\varphi_{\tau h}}
\newcommand{\dth}{d_{\tau h}}

\newcommand{\evh}{e_h^\varphi}
\newcommand{\edh}{e_h^d}

\newcommand{\etadh}{\eta_h^d}
\newcommand{\xidh}{\xi_h^d}

\newcommand{\etaph}{\eta^p_h}

\newcommand{\dthtilde}{\tilde{d}_{\tau h}}
\newcommand{\dthtildem}{\tilde{d}_{\tau h,m}}
\newcommand{\xidhm}{\xi^d_{h,m}}
\newcommand{\etadhm}{\eta^d_{h,m}}
\newcommand{\dthm}{d_{\tau h,m}}

\newcommand{\vthm}{\varphi_{\tau h,m}}

\newcommand{\Sth}{S_{\tau h}}
\newcommand{\jth}{j_{\tau h}}
\DeclareMathOperator*{\maxe}{max_{\varepsilon}}

\newtheorem{theorem}{Theorem}[section]
\newtheorem{definition}[theorem]{Definition}
\newtheorem{proposition}[theorem]{Proposition}
\newtheorem{corollary}[theorem]{Corollary}
\newtheorem{lemma}[theorem]{Lemma}

\newtheorem{remark}[theorem]{Remark}

\newtheorem{assumption}[theorem]{Assumption}

\begin{document}
\title[Optimal control of a non-smooth PDE-ODE system]{A priori error estimates for the space-time finite element approximation of a non-smooth optimal control problem governed by a coupled semilinear PDE-ODE system}
\author{Marita Holtmannspötter$^\dagger$}
\author{Arnd Rösch$^\dagger$}
\address[$^\dagger$]{Faculty of Mathematics, University of Duisburg-Essen, 45127 Essen, Germany}
\email{marita.holtmannspoetter@uni-due.de, arnd.roesch@uni-due.de}
\maketitle
\begin{abstract}
In this paper we investigate a priori error estimates for the space-time Galerkin finite element discretization of a simplified semilinear gradient enhanced damage model. The model equations are of a special structure as the state equation consists of an elliptic PDE which has to be fulfilled at almost all times coupled with a non-smooth, semilinear ODE that has to hold true in almost all points in space. The system is discretized by a constant discontinuous Galerkin method in time and usual conforming linear finite elements in space. For the uncontrolled equation, we prove linear convergence in time and an order of $\mathcal{O}(h^{\frac{3}{2}-\epsilon})$ for the discretization error in space. Our main result regarding the optimal control problem is the uniform convergence of $dG(0)cG(1)$-discrete controls to $\ovl\in\Heinsnullzwei$. Error estimates for the controls are established via a quadratic growth condition. Numerical experiments are added to illustrate the proven rates of convergence. 
\end{abstract}
\section{Introduction} 

In this paper, we derive a priori error estimates for the space-time finite element discretization of a simplified semilinear gradient enhanced damage model and the associated optimal control problem. To be more specific, we investigate the finite element approximation of the optimal control problem
\[J(\varphi,d,l)=\frac{1}{2}\Vert \varphi-\varphi_d\Vert^2_{\Lzweizwei}+\frac{1}{2}\Vert d-d_d\Vert^2_{\Lzweizwei}+\frac{\alpha_l}{2}\Vert l\Vert^2_{\Heinsnullzwei}\]
subject to the state equation
\begin{align}
\label{mod:nonlinmodbegin} -\alpha \Delta \varphi (t)+\beta \varphi(t)&=\beta d(t)+l(t)  \quad \text{ in } \Omega\\ 
\varphi(t)&=0 \quad \text{ on } \partial\Omega \\
 \partial_t d(t)&=\frac{1}{\delta}\max\{-\beta(d(t)-\varphi (t))-r,0\}  \quad \text{ a. e. in } \Omega\label{mod:nonlinmodode}\\
d(0)&=d_0 \label{mod:nonlinmodend}
\end{align}
for almost all $t\in I=[\I]$ where $l$ acts as a control and $\varphi$ and $d$ are the resulting states. A precise formulation is given in the later sections. For the discretization of the state equation we will use a discontinuous piecewise constant finite element method in time and usual $H^1$-conforming linear finite elements in space. The state equation is motivated by a specific gradient enhanced damage model, first developed in \cite{DH08,DH11} and thoroughly analyzed from a mathematical point of view in \cite{MS16i,MS16ii}. First of all, this model describes the displacement of a body $\Omega$ influenced by a given force $l$. In addition, the model features two damage variables $\varphi$ and $d$ where the first one is more regular in space whereas the second one carries the evolution of damage in time. Both are coupled by a penalty term in the free energy functional with $\beta$ being the penalty parameter. The parameter $\alpha$ originates from the gradient enhancement while $\delta$ is a viscosity parameter (see \cite{MS16i} for details). The resulting system consists of two nonlinear PDEs which have to hold true in almost all time points and an ODE that should be fulfilled in almost every point in space. All three equations are fully coupled with each other. For a first analysis of the discretization of such a model we simplified the underlying PDE system, skipping the displacement variable $u$ as well as the nonlinear material function. A further simplified linear version of this model, that lacks the $\max$-operator in the ODE, has been studied in a companion paper, see \cite{HRV18}.  The original damage model will be subject of later work.\\

As its linear counterpart, the semilinear model problem still has the special structure of the original damage model which differs from other coupled PDE-ODE-systems examined in literature. In contrast to the linear model problem studied in \cite{HRV18}, the semilinear optimal control problem is non-smooth since the $\max$-operator is not differentiable. \\

The optimal control problem is formulated with a tracking-type functional. We employ a regularization of the control in $\Heinsnullzwei=\{l\in\Heinszwei : l(0)=0\}$. An alternative and perhaps more naive choice would be a regularization in $\Lzweizwei$. But $\Lzweizwei$ as control space has two major drawbacks. First of all, the standard proof for the existence of an optimal control fails for the control-to-state operator $S\colon \Lzweizwei\to \LzweiHeins\times\Heinszwei, S(l)=(\varphi,d),$ since the Nemytskii-operator $\max \colon \Lzwo \to \Lzwo$ and thus also $S$ are not weakly continuous. And second, even if one is able to proof the existence of an optimal control $\ovl\in\Lzweizwei$, one lacks temporal regularity of the state $\varphi$. The state $\varphi$ is only as regular in time as the right-hand side of the PDE \eqref{mod:nonlinmodbegin}. For $l\in\Lzweizwei$, we only have $\varphi\in\LzweiHeins$ which is not sufficient for the derivation of temporal error estimates. For the linear model problem, we encountered the very same problem, see \cite{HRV18}. In the linear case, one is able to establish the required regularity for the optimal state, that is $\ovv\in\Heinszwei$ by employing a bootstrapping argument to the optimality system. For the semilinear model problem \eqref{mod:nonlinmodbegin}-\eqref{mod:nonlinmodend}, a similar argumentation is not promising since the optimality system for a non-smooth problem (cf. Theorem \ref{thm:optimalitysystemforP}) has a different structure and includes an additional multiplier which itself has only low temporal regularity. Both issues are solved if we regularize in $\Heinsnullzwei$. \\

The aim of this paper is twofold. First, we establish a priori discretization error estimates for the finite element discretization of \eqref{mod:nonlinmodbegin}-\eqref{mod:nonlinmodend}. The main ingredient will be the higher spatial regularity of the state $d$.
Based on these error estimates for the states, we derive error estimates for the discretization of controls via quadratic growth conditions. The main result will be the uniform convergence of the discrete controls. We focus on the same discretization technique for the control as for the states, meaning piecewise constant, discontinuous finite elements in time and $H^1$-conforming finite elements in space. Since $l\in\Heinsnullzwei$, the convergence of discrete states to continuous states requires some careful investigation and adaptation of known strategies. \\

Let us have a look at related work: There are quite a few contributions available regarding the optimal control of coupled PDE-ODE systems, cf. \cite{KGH18,BKR17,MST16,KG16,KG15,CPetal2009} and the references therein. The authors mainly focus on the analysis of their specific model and the derivation of first order necessary optimality conditions and provide tailored algorithms for the numerical solution of the optimal control problems. They do not derive discretization error estimates. In \cite{HV03}, the authors deal with the optimal control of laser surface hardening of steel and provide error estimates for a POD Galerkin approximation of the state equation. Error estimates for the optimal control of a coupled PDE-ODE system describing the velocity tracking problem for the evolutionary Navier–Stokes equations are derived in \cite{CC12,CC16} as well as companion papers. Here, the authors require a coupling of the discretization parameters in time and space for the well-posedness of their discretization technique. We emphasize, that in our contribution the discretization parameters can be chosen independently of one another. Our discretization setting is closely related to the techniques analyzed in \cite{MV08,NV11,MV18} for the space-time discretization of linear and semilinear parabolic optimal control problems, respectively. In these contributions the optimal control problem is not constrained by a coupled PDE-ODE system but rather by a single parabolic PDE. Therefore only one variable which carries the evolution in both space and time is considered. Moreover, all contributions referenced so far focus on control problems with smooth state equation. There are much less results regarding error estimates for the optimal control of non-smooth equations available. Here, we want to mention the results concerning error estimates for the optimal control of the obstacle problem \cite{MT13,HKP20} and the references therein.   Error estimates for uncontrolled parabolic equations are given in \cite{EJT85,EJ91,EJ95}. \\

The paper is organized as follows: In Section 2 we state the exact setting of the model problem and present results regarding the continuous PDE-ODE system. In Section 3 we focus on the discretization of the semilinear model problem and prove linear convergence in time and a convergence rate of $\mathcal{O}(h^{\frac{3}{2}-\varepsilon})$ in space of our discretization. Section 4 deals with error estimates for the corresponding control problem. The last section presents numerical examples. \\ 
\section{Properties of the coupled PDE-ODE system}
In this section we establish the principal assumptions on the data, some notation and the basic properties of the coupled PDE-ODE system. \\

Throughout this paper, let $\Omega\subset \mathbb{R}^N,N\in\{2,3\}$, be a convex polygonal domain with boundary $\partial \Omega$ and let $T>0$ be a given real number. The time interval will be denoted by $I:=(\I)$. Moreover, let $\alpha,\beta,\delta,r>0$ be given parameters. The initial state is, unless otherwise stated, a function in $\Lzwo$. The first state $\varphi$ is an element of the state space $V:=L^2(\I;\Heinsnull)$. The second state $d$ should belong to $X:=H^1(\I;\Lzwo)$. The control space is given as $\Heinsnullzwei\coloneqq \{l\in\Heinszwei \colon l(0)=0\}$. We use the $\Heinszwei$-seminorm as the norm on $\Heinsnullzwei$
\[\Vert l\Vert_{\Heinsnullzwei}=\Vert\partial_t l\Vert_{\Lzweizwei}.\] 

Let us state some results regarding the $\max$-operator: \begin{lemma}\label{lem:propertiesmax}
\begin{itemize} 
\item[(i)] The Nemytskii-operator $\max:\Lzwo\to\Lzwo$ associated to $\max:\mathbb{R}\to\mathbb{R}, \max(y)=\max\{y,0\}$ is well-defined and globally Lipschitz continuous with Lipschitz constant 1. 
\item[(ii)] The Nemytskii-operator $\max:\Lzweizwei\to\Lzweizwei$ associated to $\max:\Lzwo\to\Lzwo$ is well-defined and globally Lipschitz continuous with Lipschitz constant 1. 
\item[(iii)] For arbitrary $v,r\in\mathbb{R}, r>0$ we have $\max(v-r)\leq \max(v)\leq \vert v\vert$. This inequality also holds true for $v\in\Lzwo$ and $r\in\mathbb{R},r>0$, that is we have
\begin{equation} \label{eq:propmax3}
\Vert\max(v-r)\Vert_{\Lzwo}\leq\Vert\max(v)\Vert_{\Lzwo}\leq \Vert v\Vert_{\Lzwo}.
\end{equation}
\item[(iv)] The Nemytskii-operator $\max: \Hs \to \Hs$ is well-defined if and only if $0\leq s < \frac{3}{2}$. For $y\in\Hs$, we have
\begin{equation} \label{eq: boundedmaxHs}
\Vert \max(y)\Vert_{\Hs}\leq C\Vert y\Vert_{\Hs}.
\end{equation}
\end{itemize}
\end{lemma} 
\begin{proof}
Items (i) and (ii) are proven for $L^\infty(\Omega)$ and $L^2(\I;L^\infty(\Omega))$ in \cite{S17}, section 5. The proof for our case is identical. Item (iii) holds true for $\max: \mathbb{R}\to \mathbb{R}$ and therefore directly transfers to $\Lzwo$. This leaves item (iv): In \cite{RS96}, section 5.4, the assertion is proven for the absolute value function which is equivalent to the $\max$-function. 
\end{proof}
We use the following short notation for inner products and norms on $\Lzwo$ and $\Lzweizwei$:
\begin{alignat*}{2}
(v,w)&:=(v,w)_{\Lzwo}, \qquad (v,w)_{\ito} &&:=(v,w)_{\Lzweizwei}, \\
\Vert v\Vert &:=\Vert v\Vert_{\Lzwo}, \qquad \qquad \Vert v\Vert_{\ito} &&:= \Vert v\Vert_{\Lzweizwei}. \\
\end{alignat*}
Instead of the (strong) formulation \eqref{mod:nonlinmodbegin}-\eqref{mod:nonlinmodend} we will work with the weak formulation of the problem. We define the bilinear form $B$ 
\begin{equation} \label{eq:bilinearformnonlinear}
B((\varphi,d),(\psi,\lambda))=\alpha(\nabla\varphi,\nabla\psi)_{\ito}+\beta(\varphi-d,\psi)_{\ito}+(\partial_t d,\lambda)_{\ito}.
\end{equation}
Then, the weak formulation reads as follows: Find states $(\varphi,d)\in V\times X$ satisfying
\begin{equation} \label{eq:nonlinearweakformulationcont}
B((\varphi,d),(\psi,\lambda))=(l,\psi)_{\ito}+\frac{1}{\delta}(\max(-\beta(d-\varphi)-r),\lambda)_{\ito} \quad \forall (\psi,\lambda)\in V\times X
\end{equation}
and the initial value condition $d(0)=d_0$. \\

We start with the investigation of the continuous problem. Our first result covers the unique solvability  of the semilinear variational problem \eqref{eq:nonlinearweakformulationcont}.
\begin{proposition} \label{prop:existencesolutionnonlinear}
For a fixed right-hand side $l\in\Lzweizwei$ and initial state $d_0\in\Lzwo$ there exists a unique solution $(\varphi,d)\in V\times X$ of equation \eqref{eq:nonlinearweakformulationcont}. Moreover, the solution exhibits the improved regularity
\begin{align*}
\varphi \in &L^2(\I; H^2(\Omega)\cap H^1_0(\Omega)) \\
d \in &\Heinszwei \hookrightarrow C(\overline{I},\Lzwo).
\end{align*}
\end{proposition}
\begin{proof}
The proposition can be proven analogous to \cite{HRV18}, Proposition 3.1, since $\max:\Lzwo\to\Lzwo$ is Lipschitz continuous with constant 1. 
\end{proof}
For later references, we denote by $\Phi: \Lzwo\times\Lzwo\to\Heinsnull$, $\Phi: (l,d)\mapsto \varphi$, the solution operator of the elliptic PDE
\begin{equation} \label{eq:weakvarphicont}
\alpha(\nabla\varphi,\nabla \psi)+\beta(\varphi,\psi)=(\beta d+l,\psi) \quad \forall \psi\in \Heinsnull.
\end{equation}
\begin{lemma}
If we assume $l\in\Heinszwei$ then $\varphi$ is partial differentiable with respect to time and we have $\partial_t\varphi=\Phi(\partial_t l,\partial_t d)$ and $\varphi\in\Heinszwei$. We will require this regularity for the temporal error estimation.
\end{lemma}
\begin{lemma} \label{lem:boundednessstatescontinuous}
In addition to the assumptions of Proposition \ref{prop:existencesolutionnonlinear}, let $l\in\Heinszwei$ hold true. Then, the solution $(\varphi,d)\in V\times X$ fulfills the stability estimates
\begin{equation}
\Vert d\Vert_{L^\infty(\I;\Lzwo)}\leq C\{\Vert d_0\Vert+\Vert l\Vert_{L^1(\I;\Lzwo)}\}
\end{equation}
\begin{equation}
\Vert \nabla^2\varphi\Vert_{\ito}+\Vert \nabla \varphi\Vert_{\ito}+\Vert\varphi\Vert_{\ito}+\Vert d\Vert_{\ito}+\Vert\partial_t d\Vert_{\ito}\leq C\{\Vert d_0\Vert+\Vert l\Vert_{\ito}\}
\end{equation}
\begin{equation}
\Vert\partial_t\varphi\Vert_{\ito}\leq C\{\Vert d_0\Vert+\Vert l\Vert_{\Heinszwei}\}
\end{equation}
with a constant $C>0$. 
\end{lemma}
\begin{proof}
The first assertion follows with Gronwall's inequality. The second assertion may be proven with standard techniques. The third stability estimate is a consequence of the definition of $\partial_t\varphi$ and the estimates from the second assertion. 
\end{proof}
Regarding the optimal control, we need a slightly more general existence result for controls $l\in\LzweiHminuseins$.
\begin{lemma} \label{lem:existencelzweihminuseins}
For a fixed right-hand side $l\in\LzweiHminuseins$ and initial state $d_0\in\Lzwo$, there exists a unique solution $(\varphi,d)\in V\times X$. The control-to-state operator $S:\LzweiHminuseins $ $\to V\times X$ is Lipschitz continuous, that is there exists a constant $L_S>0$ such that
\begin{equation}
\Vert S(l_1)-S(l_2)\Vert_{V\times X}\leq L_S\Vert l_1-l_2\Vert_{\LzweiHminuseins}
\end{equation}
for all $l_1,l_2\in\LzweiHminuseins$.
\end{lemma}
\begin{proof}
The existence of a unique solution follows as in the previous proposition. This leaves to prove the Lipschitz continuity. Thus, we denote by $(\varphi_1,d_1)=S(l_1)$ and $(\varphi_2,d_2)=S(l_2)$ the solution of the state equation for two different right-hand sides $l_1,l_2\in\LzweiHminuseins$. As $\Phi$ is known to be Lipschitz continuous, we have
\begin{equation} \label{eq:existencelzweihminuseinsHR1}
\Vert (\varphi_1-\varphi_2)(t)\Vert_{\Heins}\leq L_\Phi\{\Vert (d_1-d_2)(t)\Vert+ \Vert (l_1-l_2)(t)\Vert_{\Hminuseins}\}.
\end{equation}
Therefore, it suffices to prove the Lipschitz continuity with respect to $d$. We subtract the reduced ODEs to arrive at
\[\partial_t (d_1-d_2)(t)=\frac{1}{\delta}\left(\max(-\beta(d_1(t)-\varphi_1(t))-r)-\max(-\beta(d_2(t)-\varphi_2(t))-r)\right).\] Lipschitz continuity can now be achieved by means of \cite{EE04}, Thm. 7.5.3.
\end{proof}

\section{Numerical Analysis of the semilinear model equation} \label{NAforstateeq}
This section is devoted to the error estimation for the discretization of the coupled PDE-ODE system. 
\subsection{Semidiscretization in time}
For the discretization in time we will employ discontinuous constant finite elements. Therefore, we consider a partition of the time interval $\overline{I}=[\I]$ as
\[\overline{I}=\{0\}\cup I_1\cup ... \cup I_M\]
with subintervals $I_m=(t_{m-1},t_m]$ of length $\tau_m$ and time points
\[0=t_0<t_1<...<t_{M-1}<t_M=T.\]
We set $\tau:=\max \{\tau_m : m=1,...,M\}$. The semidiscrete trial and test spaces are given as
\[\Vtau:=\{v_\tau\in V: v_{\tau\vert_{I_m}}\in\mathbb{P}_0(I_m;\Heinsnull), m=1,...,M\},\]
\[\Xtau:=\{d_\tau\in \Lzweizwei : d_{\tau\vert_{I_m}}\in\mathbb{P}_0(I_m;\Lzwo), m=1,...,M\}.\]
Note, that $\Vtau\subset V$ but $\Xtau\not\subset X$. Moreover, $\Vtau$ is dense in $\Xtau$ due to the dense embedding of $\Heinsnull\overset{\text{d}}{\hookrightarrow}\Lzwo$. We use the notation
\[(v,w)_{\imto}:=(v,w)_{L^2(I_m;\Lzwo)} \qquad \text{ and } \qquad \Vert v\Vert_{\imto} := \Vert v\Vert_{L^2(I_m;\Lzwo)}.\]
To express the jumps possibly occurring at the nodes $t_m$ we define
\[v^+_{\tau,m}:=\lim\limits_{t\rightarrow 0^+} v_{\tau}(t_m+t), \quad v^-_{\tau,m}:=\lim\limits_{t\rightarrow 0^+} v_{\tau}(t_m-t)=v_\tau(t_m), \quad [v_\tau]_m=v^+_{\tau,m}-v^-_{\tau,m}.\]
Note, that for functions piecewise constant in time the definition reduces to
\[v^+_{\tau,m}=v_\tau(t_{m+1})=:v_{\tau,m+1}, \qquad v^-_{\tau,m}=v_{\tau}(t_m)=: v_{\tau,m},\quad [v_\tau]_m=v_{\tau,m+1}-v_{\tau,m}.\]
The semidiscrete bilinear form $\B$ is given as
\begin{align*}
\B((\vt,\dt),(\psi,\lambda))&=\alpha(\nabla\vt,\nabla\psi)_{\ito}+\beta(\vt,\psi)_{\ito}-\beta(\dt,\psi)_{\ito}\\ &+\sum\limits_{m=1}^M(\partial_t\dt,\lambda)_{I_m}+\sum\limits_{m=2}^M([\dt]_{m-1},\lambda_{m-1}^+)+(d^+_{\tau,0},\lambda_0^+).
\end{align*}
Then, the semidiscrete semilinear problem is given as follows: Find states $(\vt,\dt)\in\Vtau\times\Xtau$ such that
\begin{equation} \label{eq:semidiscreteprimalproblemnonlinear}
\B((\vt,\dt),(\psi,\lambda))=(l,\psi)_{\ito}+\frac{1}{\delta}(\max(-\beta(\dt-\vt)-r),\lambda)_{\ito}+(d_0,\lambda_0^+) 
\end{equation}
is fulfilled for all $(\psi,\lambda)\in\Vtau\times\Xtau$. \\

We will require the interpolation/projection onto $\Xtau$ and $\Vtau$, respectively.  Therefore, we define the semidiscrete interpolation operator $I_\tau : C(\overline{I};\Lzwo)\to\Xtau$ with $I_\tau d_{\vert_ {I_m}}\in \mathbb{P}_0(I_m;\Lzwo)$ via $(I_\tau d)(t_m)=d(t_m)$ for $m=1,...,M$. For the projection we employ the standard $L^2$-projection in time $P_\tau : \Lzweizwei\to\Xtau$ given by $P_\tau\varphi_{\vert_{I_m}}:=\frac{1}{\tau_m}\int\limits_{I_m} \varphi(t)dt$. Both operators will always be denoted by the same symbols despite possibly different domains and ranges. Note, that if $\varphi\in V$ then $P_\tau\varphi\in\Vtau$ as integration in time preserves the spatial regularity due to the definition of the Bochner integral. In particular, we have \begin{equation} \label{eq:nablaPtvarphi}(\varphi-P_\tau\varphi,\psi)_{I\times\Omega}=(\nabla \varphi-\nabla P_\tau\varphi,\nabla\psi)_{I\times\Omega}=0\end{equation} for any $\psi\in\Vtau$. \\

Having introduced all necessary notation, we have a look at the unique solvability of the semidiscrete problem next.
\begin{proposition} \label{prop:existencesemidiscretestatenonlinear}
Let $l\in\Lzweizwei$ and $d_0\in\Lzwo$ be given. Then, the semidiscrete semilinear problem \eqref{eq:semidiscreteprimalproblemnonlinear} possesses a unique solution $(\vt,\dt)\in\Vtau\times\Xtau$ provided that $\tau$ is chosen small enough. 
\end{proposition}
\begin{proof}
Similar to the linear case (see \cite{HRV18}, Prop. 3.3), the unique solution $\vt\in L^\infty(\I;\Heinsnull)$ is given as $\vt(t,x)=\sum\limits_{m=1}^M \varphi_{\tau ,m}(x)\chi_{I_m}(t)$ with $\varphi_{\tau ,m}=\Phi(P_\tau l_{\vert I_m},d_{\tau ,m})$ while one can prove the existence of a unique $\dt\in L^\infty(\I;\Lzwo)$ by the application of Banach's fixed point theorem to the reduced fixed point equation in $\Lzwo$
\begin{equation} \label{eq:semidiscretefixpointeqnonlinearoneinterval}
d_{\tau,m}=F_m(d_{\tau,m})=d_{\tau,m-1}+\frac{\tau_m}{\delta}\max(-\beta(d_{\tau,m}-\Phi(P_\tau l_{\vert I_m},d_{\tau,m}))-r)
\end{equation}
on each subinterval $I_m,m=1,...,M$ starting with $d_{\tau,0}=d_0$. Due to the Lipschitz continuity of $\Phi$ and $\max$ it is easy to prove that $F_m:\Lzwo\to\Lzwo$ is a contraction if $\frac{\beta}{\delta}\tau_m(1+L_\Phi)<1$ for all $m=1,\ldots,M$. 
\end{proof}

Our next results cover the stability of the semidiscrete solution.
\begin{lemma} \label{lem:stabestimateinftyzwo}
For the solution $(\vt,\dt)\in\Vtau\times\Xtau$ of the semidiscrete state equation \eqref{eq:semidiscreteprimalproblemnonlinear} with right-hand side $l\in\Lzweizwei$ and initial state $d_0\in\Lzwo$ the stability estimate
\begin{equation}
\label{eq:stabestimateinftytwo}
\Vert\dt\Vert_{\Linfzwo}\leq C\{\Vert d_0\Vert+\Vert l\Vert_{\LL{1}{2}}\}
\end{equation}
holds true with a constant $C>0$ independent of $\tau$ provided that $\tau$ is small enough. 
\end{lemma}
\begin{proof}
At first, we take norms on both sides of \eqref{eq:semidiscretefixpointeqnonlinearoneinterval}. Then, the triangle inequality, \eqref{eq:propmax3}, the Lipschitz continuity of $\Phi$ and $\Phi(0,0)=0$ yield
\begin{align*}
\Vert d_{\tau,m}\Vert&\leq \Vert d_{\tau,m-1}\Vert+ \frac{\tau_m}{\delta}\beta(1+L_\Phi)\Vert d_{\tau,m}\Vert+\frac{\beta}{\delta}L_\Phi\int\limits_{I_m}\Vert l(t)\Vert_{\Lzwo} \,dt.
\end{align*}
With the assumption $\tau_m\frac{\beta}{\delta}(1+L_\Phi)<1$, we arrive at
\[\Vert d_{\tau,m}\Vert \leq \frac{1}{1-\tau_m\frac{\beta}{\delta}(1+L_\Phi)}\left(\Vert d_{\tau,m-1}\Vert+\frac{\beta}{\delta}L_\Phi\Vert l\Vert_{L^1(I_m;\Lzwo)}\right). \]
Induction leads to
\begin{align}
\Vert d_{\tau,m}\Vert&\leq \prod\limits_{j=1}^M\frac{1}{1-\tau_j\frac{\beta}{\delta}(1+L_\Phi)}\left(\Vert d_0\Vert+\frac{\beta}{\delta}L_\Phi\Vert l\Vert_{L^1(I;\Lzwo)}\right). \label{eq:abschdtmzwischenerg}
\end{align}
Next, due to 
\[\prod\limits_{j=1}^M\frac{1}{1-\tau_j\frac{\beta}{\delta}(1+L_\Phi)}\leq \exp\left(\frac{\frac{\beta}{\delta}(1+L_\Phi)}{1-\tau\frac{\beta}{\delta}(1+L_\Phi)}\sum\limits_{j=1}^M\tau_j\right)=\exp\left(\frac{\frac{\beta}{\delta}(1+L_\Phi)}{1-\tau\frac{\beta}{\delta}(1+L_\Phi)}T\right)\]
$\prod\limits_{j=1}^M\frac{1}{1-\tau_j\frac{\beta}{\delta}(1+L_\Phi)}$ is bounded from above by a sequence converging in $\tau$. Thus, there exists an upper bound $C>0$ independent of $\tau$. If we insert this upper bound into \eqref{eq:abschdtmzwischenerg}, we get the final estimate
\[\Vert d_{\tau,m}\Vert \leq C\left(\Vert d_0\Vert+\frac{\beta}{\delta}L_\Phi\Vert l\Vert_{L^1(\I;\Lzwo)}\right).\]
Since the constant is independent of $\tau$ and $m$, this finishes the proof.
\end{proof}
\begin{theorem} \label{thm:stabestimatesprimalnonlinear}
For the solution $(\vt,\dt)\in\Vtau\times\Xtau$ of the semidiscrete state equation \eqref{eq:semidiscreteprimalproblemnonlinear} with right-hand side $l\in\Lzweizwei$ and initial state $d_0\in\Lzwo$ the stability estimate
\begin{equation}
\label{eq:stabestimatedtwotwo}
\Vert\Delta \vt\Vert^2_{\ito}+\Vert\nabla\vt\Vert^2_{\ito}+\Vert\vt\Vert^2_{\ito}+\Vert\dt\Vert_{{\ito}}^2
+\sum\limits_{m=1}^M\tau_m^{-1}\Vert[\dt]_{m-1}\Vert^2 \leq C\{\Vert d_0\Vert^2+\Vert l\Vert_{\ito}^2\}
\end{equation}
holds true with a constant $C>0$ independent of $\tau$, provided that $\tau$ is small enough. 
\end{theorem}
\begin{proof}
The assertion may be proven as in the linear case, see \cite{HRV18}, Thm 3.7, because $\max\colon\Lzwo\to\Lzwo$ is Lipschitz continuous with constant 1. 
\end{proof}

We will make use of an auxiliary (dual) equation of the following form: Find dual states $(\zt,\pt)\in\Vtau\times\Xtau$ such that
\begin{equation}
\label{eq:semidiscretedualstatenonlinear}
\B((\psi,\lambda),(\zt,\pt))+(\lambda-\psi,f\pt)_{\ito}=(\psi,g_1)_{\ito}+(\lambda,g_2)_{\ito}+(p_T,\lambda_M^-)
\end{equation}
is fulfilled for all $(\psi,\lambda)\in\Vtau\times\Xtau$. For the moment, we only assume that $g_1,g_2\in\Lzweizwei$ and $p_T\in\Lzwo$ are given data. The function $f$ should belong to $L^\infty(I\times\Omega)$ with $\vert f(t,x)\vert\leq\frac{\beta}{\delta}$ for almost all $(t,x)\in I\times\Omega$. We will later use a specific function $f$ that satisfies these assumptions. The dual representation of $\B$ is
\begin{align} \label{eq:Bdualnonlinear}
\B((\psi,\lambda),(\zt,\pt))=&\alpha(\nabla\zt,\nabla\psi)_{\ito}+\beta(\zt,\psi)_{\ito}-\beta(\lambda,\zt)_{\ito} \\ \nonumber
&-\sum\limits_{m=1}^M(\partial_t\pt,\lambda)_{\imto}-\sum\limits_{m=1}^{M-1}([\pt]_m,\lambda_m^-)+(p_{\tau,M}^-,\lambda_M^-).
\end{align}
Regarding the existence of a unique solution, we have the following result
\begin{proposition} \label{cor: stabestimatedualtimenonlinear}
The auxiliary dual equation \eqref{eq:semidiscretedualstatenonlinear} possesses a unique solution $(\zt,\pt)\in\Vtau\times\Xtau$ for given data $g_1,g_2\in\Lzweizwei$, $p_T\in\Lzwo$ and $f\in L^\infty(I\times\Omega)$ with $\vert f(t,x)\vert\leq\frac{\beta}{\delta}$ for almost all $(t,x)\in I\times\Omega$, provided that $\tau$ is sufficiently small.
\end{proposition}
\begin{proof}
The idea of the proof is similar to the proof of Proposition \ref{prop:existencesemidiscretestatenonlinear}. First of all, $\tilde{f}_m(x):=\int\limits_{I_m} f(t,x)dt\in L^\infty(\Omega)$ for all $m=1,\ldots,M$ and for a given $\pt\in\Xtau\subset L^\infty(\I;\Lzwo)$ the function $f\pt$ belongs to $\Lzweizwei$. Thus, $P_\tau(f\pt)_{\vert I_m}\in\Lzwo$ is well defined. It is easy to check that the unique solution of the PDE is given as $\zt=\sum\limits_{m=1}^M z_{\tau ,m}\chi_{I_m}(t)$ with $z_{\tau ,m}=\Phi(P_\tau g_{1\vert I_m},\frac{1}{\beta}P_\tau(f\pt)_{\vert I_m})\in\Hzwei\cap\Heinsnull$. The existence of a unique $\pt\in\Xtau$ may now be concluded by applying Banach's fixed point theorem to the reduced fixed point equation
\begin{equation}
p_{\tau,m}=p_{\tau,m+1}-p_{\tau,m}\tilde{f}_m+\beta\tau_m\Phi(P_\tau g_{1\vert I_m},\frac{1}{\beta}P_\tau(f\pt)_{\vert I_m})+\int\limits_{I_m} g_2(t)\,dt=: F_m(p_{\tau ,m}).
\end{equation} 
$F_m:\Lzwo\to\Lzwo$ is a contraction if $\frac{\beta}{\delta}\tau_m(1+L_\Phi)<1$.
\end{proof}
Stability estimates may now be proven with similar arguments as for the primal states:
\begin{corollary}\label{cor:stabestimatesdualnonlinear}
For the solution $(\zt,\pt)\in\Vtau\times\Xtau$ of the semidiscrete dual equation \eqref{eq:semidiscretedualstatenonlinear}, the stability estimates
\begin{equation} \label{eq:stabestimatedualnonlinearinftyzwo}
\Vert \pt\Vert_{L^\infty(\I;\Lzwo)}\leq C\{\Vert p_T\Vert + \Vert g_1\Vert_{\ito}+\Vert g_2\Vert_{\ito}\}
\end{equation}
\begin{equation}
\label{eq:stabestimatedualnonlineartwotwo}
\Vert\Delta \zt\Vert^2_{\ito}+\Vert\nabla\zt\Vert^2_{\ito}+\Vert\zt\Vert^2_{\ito}+\Vert\pt\Vert_{{\ito}}^2
+\sum\limits_{m=1}^M\tau_m^{-1}\Vert[\pt]_{m}\Vert^2 \leq C\{\Vert p_T\Vert^2+\Vert g_1\Vert_{\ito}^2+\Vert g_2\Vert_{\ito}^2\}
\end{equation}
hold true with a constant $C>0$ independent of $\tau$, provided that $\tau$ is small enough.  The jump term $[p_\tau]_M$ is defined as $p_T-p_{\tau,M}^-$. 
\end{corollary}

Before we state the main result regarding the temporal error, we require a property which will be referred to as Galerkin orthogonality in time, namely
\begin{equation} \label{eq:Galerkinorthotimenonlinear}
\B((\varphi-\vt,d-\dt),(\psi,\lambda))=\frac{1}{\delta}(\max(-\beta(d-\varphi)-r)-\max(-\beta(\dt-\vt)-r),\lambda)_{\ito} 
\end{equation}
holds true for all $(\psi,\lambda)\in\Vtau\times\Xtau$. With the Galerkin orthogonality at hand and a specific choice for $f$ in the dual equation, we may now prove the main result of this subsection:
\begin{theorem}
\label{thm:temporalerrornonlinearstate}
Let $l\in\Heinszwei$ and $d_0\in\Lzwo$ be fulfilled. For the errors $\evt:=\varphi-\vt$ and $\edt:=d-\dt$ between the continuous solution $(\varphi,d)\in V\times X$ of \eqref{eq:nonlinearweakformulationcont} and the dG(0) semidiscretized solution $(\vt,\dt)\in\Vtau\times\Xtau$ of \eqref{eq:semidiscreteprimalproblemnonlinear}, we have the error estimate
\[\Vert \evt\Vert_{\ito}+\Vert \edt\Vert_{\ito}\leq C\tau\{\Vert \partial_t\varphi\Vert_{\ito}+\Vert\partial_t d\Vert_{\ito}\}\]
with a constant $C>0$ independent of the temporal discretization parameter $\tau$. 
\end{theorem}
\begin{proof} We will prove the theorem with arguments used in \cite{NV11}, Thm. 3.3,  to show their corresponding result regarding the temporal error estimate.
Consider the dual equation \eqref{eq:semidiscretedualstatenonlinear} with $f:[\I]\times\Omega\to\mathbb{R}$ defined as
\[f(t,x)=\begin{cases}0 &, \text{ if } \edt(t,x)-\evt(t,x)=0, \\ \frac{\frac{1}{\delta}\max(-\beta(\dt(t,x)-\vt(t,x))-r)-\frac{1}{\delta}\max(-\beta(d(t,x)-\varphi(t,x))-r)}{d(t,x)-\varphi(t,x)-(\dt(t,x)-\vt(t,x))} &,\text{ else. } \end{cases}\]
The Lipschitz continuity of $\max: \mathbb{R}\to\mathbb{R}$ yields $f\in L^\infty(I\times\Omega)$ and $\vert f(t,x)\vert\leq\frac{\beta}{\delta}$. Thus, the dual equation possesses a unique solution which satisfies the stability estimates \eqref{eq:stabestimatedualnonlineartwotwo}. We will split the temporal errors 
\[\evt=\varphi-\vt=\underbrace{\varphi-P_\tau\varphi}_{=\etavt}+\underbrace{P_\tau\varphi-\vt}_{=\xivt}, \quad \edt=d-\dt=\underbrace{d-I_\tau d}_{=\etadt}+\underbrace{I_\tau d-\dt}_{=\xidt}.\]
The choice $g_1=\evt$, $g_2=\edt$ and $p_T=0$ in \eqref{eq:semidiscretedualstatenonlinear} together with the definition of $f$ and the Galerkin orthogonality \eqref{eq:Galerkinorthotimenonlinear} leads to 
\begin{align*}
\Vert\evt\Vert^2_{\ito}+\Vert\edt\Vert^2_{\ito}&=(\xivt,\evt)_{\ito}+(\xidt,\edt)_{\ito}+(\etavt,\evt)_{\ito}+(\etadt,\edt)_{\ito} \\
&=\B((\xivt,\xidt),(\zt,\pt))+(\xidt-\xivt,f\pt)_{\ito}+(\etavt,\evt)_{\ito}+(\etadt,\edt)_{\ito} \\
&=-\B((\etavt,\etadt),(\zt,\pt))-(\etadt-\etavt,f\pt)_I+(\etavt,\evt)_{\ito}+(\etadt,\edt)_{\ito} .
\end{align*}
For the first term, we have
\begin{align*}
\B((\etavt,\etadt),(\zt,\pt))&=-\beta(\etadt,\zt)_{\ito}.\\
\end{align*}
The other terms vanish due to the properties of the interpolation operator $I_\tau$, the projection operator $P_\tau$ and due to \eqref{eq:nablaPtvarphi}. The application of Cauchy-Schwarz' inequality and the stability estimates for dual equations \tr{from Corollary \ref{cor:stabestimatesdualnonlinear}} yield
\begin{align*}
\Vert\evt\Vert^2_I+\Vert\edt\Vert^2_I&\leq C(\Vert\etavt\Vert_I+\Vert\etadt\Vert_I)(\Vert\evt\Vert_I+\Vert\edt\Vert_I).
\end{align*}
From here the assertion is obtained with known error estimates for the interpolation operator $I_\tau$ and the projection operator $P_\tau$. 
\end{proof}
\subsection{Discretization in space} \label{numanaspatial}
We now turn our attention to the space-time discretization of our problem. We use $H^1$-conforming linear finite elements in space. Thus, we consider a quasi-uniform mesh $\mathbb{T}_h$ of shape regular triangles $\mathcal{T}$, which do not overlap and cover the domain $\Omega$. By $h_\mathcal{T}$ we denote the size of the triangle $\mathcal{T}$ and $h$ is the maximal triangle size. On the mesh $\mathbb{T}_h$ we construct two conforming finite element spaces 
\[\Vh=\{v\in C(\overline{\Omega}) : v_{\vert \mathcal{T}}\in \mathbb{P}_1(\mathcal{T}), \mathcal{T}\in\mathbb{T}_h, v_{\vert\partial\Omega}=0\},\]
\[\Xh=\{v\in C(\overline{\Omega}) : v_{\vert \mathcal{T}}\in \mathbb{P}_1(\mathcal{T}), \mathcal{T}\in\mathbb{T}_h\}.\]
Then the space-time discrete finite element spaces are given by
\[\Vth=\{v\in L^2(\I,\Vh) : v_{\vert I_m}\in\mathbb{P}_0(I_m;\Vh)\}\subset \Vtau,\]
\[\Xth=\{v\in L^2(\I,\Xh) : v_{\vert I_m}\in\mathbb{P}_0(I_m;\Xh)\}\subset \Xtau.\]
The space-time discrete equation then reads as follows: Find states $(\vth,\dth)\in\Vth\times\Xth$ such that
\begin{equation}
\label{eq:discretestateequationnonlinear}
\B((\vth,\dth),(\psi,\lambda))=(l,\psi)_{\ito}+\frac{1}{\delta}(\max(-\beta(\dth-\vth)-r),\lambda)_{\ito} +(d_0,\lambda^+_0)
\end{equation}
is fulfilled for all $(\psi,\lambda)\in\Vth\times\Xth$.  Note, that although we set $X=\Heinszwei$ as the state space for $d$ we choose a piecewise linear and continuous approximation in space for $d$. This is due to the fact, that we will show higher spatial regularity of $d$ later on.\\

For the projection onto $\Vth$ and $\Xth$ we work with the standard $L^2$-projections $P^V_h\colon\Lzwo\to\Vh,P^X_h\colon\Lzwo\to\Xh$ in space on each subinterval $I_m$ and define the time-space projections $\pi^V_h\colon\Vtau\to\Vth,\pi^X_h\colon\Xtau\to\Xth$ via $(\pi^V_h z)(t)=P^V_h(z(t))$ and $(\pi^X_h z)(t)=P^X_h(z(t))$ respectively. \\

We begin with unique solvability of \eqref{eq:discretestateequationnonlinear}.
\begin{theorem}
\label{thm:existencediscretestatenonlinear} Let $l\in\LzweiHminuseins$ and $d_0\in\Lzwo$ be given. Then, the discrete state equation \eqref{eq:discretestateequationnonlinear} possesses a unique solution $(\vth,\dth)\in\Vth\times\Xth$ for $\tau$ sufficiently small.
\end{theorem}
\begin{proof}
We define the mapping $\Phi_h\colon\Hminuseins\times\Hminuseins\mapsto \Vh, \Phi_h(l,d)=\varphi_h,$ as the solution operator of the discrete version of \eqref{eq:weakvarphicont}
\[\alpha(\nabla\varphi_h,\psi)+\beta(\varphi_h,\psi)=\beta(d,\psi)+(l,\psi) \quad \forall \psi\in\Vh.\]
Then, for given $\dth\in\Xth$ and $l\in\Lzweizwei$, the function $\vth:=\sum\limits_{m=1}^M \varphi_{\tau h,m}\chi_{I_m}(t)$, $\varphi_{\tau h,m}=\Phi_h(P_\tau l_{\vert I_m},d_{\tau h,m})$, belongs to $\Vth$ and satisfies 
\[\alpha(\nabla\vth,\nabla\psi)_{\ito}+\beta(\vth,\psi)_{\ito}=(l,\psi)_{\ito}+\beta(\dth,\psi)_{\ito} \qquad \forall \psi\in\Vth.\]
Thus, it suffices to prove the existence of a unique solution $\dth\in\Xth$ of the reduced ODE 
\begin{equation}
\label{eq:reducedODEdiscretenonlinear}
\sum\limits_{m=2}^M([\dth]_{m-1},\lambda_{m-1}^+)+(d_{\tau h,1},\lambda^+_0)=\frac{1}{\delta}(\max(-\beta(\dth-\vth)-r),\lambda)_{\ito}+(d_0,\lambda^+_0) \qquad \forall \lambda\in\Xth
\end{equation}
with $\vth=\sum\limits_{m=1}^M \Phi_h(P_\tau l_{\vert I_m},d_{\tau h,m})\chi_{I_m}(t)$. The above problem is equivalent to the unique solvability of
\begin{equation} \label{eq:reducedODEdiscretenonlinearoneinterval}
(d_{\tau h,m},\lambda_m)=(d_{\tau h,m-1},\lambda_m)+\frac{\tau_m}{\delta}(\max(-\beta(d_{\tau h,m}-\Phi_h(P_\tau l_{\vert I_m},d_{\tau h,m}))-r),\lambda_m) \qquad \forall \lambda_m\in\Xh
\end{equation}
for all $m=1,\ldots,M$ with $d_{\tau h,0}=P^X_hd_0$. 
The unique solvability of this formulation can be obtained by means of Brouwer's fixed point theorem.
\end{proof}
With arguments similar to the continuous case, one obtains that the associated control-to-state operator $\Sth\colon \LzweiHminuseins\to\Vth\times\Xth$, $\Sth(l)=(\vth,\dth)$, is Lipschitz continuous with Lipschitz constant $L_S$. 

We continue with a collection of preliminary results for the error $\evh$ which solely rely on the results known for the elliptic case. 
\begin{lemma} \label{lem:preliminaryerrorvarphispatial}
Let $(\vt,\dt)\in\Vtau\times\Xtau$ be the solution of the semidiscrete state equation \eqref{eq:semidiscreteprimalproblemnonlinear} and let $(\vth,\dth)\in\Vth\times\Xth$ be the space-time discrete solution of \eqref{eq:discretestateequationnonlinear} for a given right-hand side $l\in\Lzweizwei$ and an initial state $d_0\in\Lzwo$. Then, we have the preliminary error estimates
\begin{equation}
\label{eq:preliminaryerrorvarphispatiallzwei}
\Vert\evh\Vert_{\ito}\leq C\{h^2\Vert\nabla^2\vt\Vert_{\ito}+\Vert\edh\Vert_{\ito}\}
\end{equation}
\begin{equation}
\label{eq:preliminaryerrorvarphispatialheins}
\Vert \evh\Vert_{L^2(\I;\Heins)}\leq C\{h\Vert\nabla\vt\Vert_{\ito}+\Vert\edh\Vert_{\ito}\}.
\end{equation}
for the errors $\evh=\vt-\vth$ and $\edh=\dt-\dth$ with a constant $C>0$ independent of $h$ and $\tau$. 
\end{lemma}
\begin{proof}
From the proofs of existence, we know that $\varphi_{\tau,m}=\Phi(P_\tau l_{\vert I_m},d_{\tau,m})$ and $\varphi_{\tau h,m}=$\linebreak $\Phi_h(P_\tau l_{\vert I_m},d_{\tau h,m})$ hold true for all $m=1,\ldots,M$. Thus, using the Lipschitz continuity of $\Phi$ and known error estimates from the elliptic case, we may estimate 
\begin{align*}
\Vert e^\varphi_{h,m}\Vert &\leq \Vert \Phi(P_\tau l_{\vert I_m},d_{\tau ,m})-\Phi_h(P_\tau l_{\vert I_m},d_{\tau ,m})\Vert + \Vert \Phi_h(P_\tau l_{\vert I_m},d_{\tau ,m})-\Phi_h(P_\tau l_{\vert I_m},d_{\tau h,m}) \Vert \\
&\leq Ch^2\Vert\nabla^2\varphi_{\tau,m}\Vert + L_\Phi\Vert e^d_{h,m}\Vert.
\end{align*}
Similarly, if we replace the $L^2$-norm by the $H^1$-norm we get the estimate
\[\Vert e^\varphi_{h,m}\Vert_{\Heins}\leq Ch\Vert\nabla^2\varphi_{\tau,m}\Vert + L_\Phi\Vert e^d_{h,m}\Vert.\]
Squaring and integrating in time on both sides gives the desired estimates. 
\end{proof}
Based on these preliminary results for the discretization error $\evh$, there are two possibilities how to deduce error estimates for $\edh$. The first alternative is to reduce the semidiscrete and discrete ODE onto the variable $\dt$ and $\dth$ by inserting $\varphi_{\tau,m}=\Phi(P_\tau l_{\vert I_m}, d_{\tau ,m})$ and $\varphi_{\tau h,m}=\Phi_h(P_\tau l_{\vert I_m}, d_{\tau h ,m})$, respectively. For this approach, the previous results regarding stability of primal and dual solutions cannot be employed directly and have to be adjusted to the new situation. Once error estimates for the spatial error $\edh$ are derived, Lemma \ref{lem:preliminaryerrorvarphispatial} gives us the discretization error estimates also for $\evh$ in $\Heins$ and $\Lzwo$. The second alternative is to continue to work with the coupled system and derive estimates for $\evh$ and $\edh$ simultaneously. Both alternatives will result in the same order of convergence for the errors $\evh$ and $\edh$. We choose the first alternative as this approach is more general and allows for an adoption to the original damage model.\\

We define the reduced bilinear form $\bb: \Xtau\times \Xtau \to \mathbb{R}$
\begin{equation} \label{def:reducedbilinearform}
\bb(\dt,\lambda)=\sum\limits_{m=2}^M ([\dt]_{m-1},\lambda_{m-1}^+)+(d_{\tau,0}^+,\lambda_0^+).
\end{equation}
The reduced semidiscretized ODE can be formulated as
\begin{equation} \label{eq:reducedsemidiscreteODE}
\bb(\dt,\lambda)=\frac{1}{\delta}(\max(-\beta(\dt-\vt)-r),\lambda)_{\ito}+(d_0,\lambda_0^+) \qquad \forall \lambda\in\Xtau
\end{equation}
while the space-time discretized ODE is given as (compare \eqref{eq:reducedODEdiscretenonlinear})
\begin{equation}
\label{eq:reduceddiscreteODE}
\bb(\dth,\lambda)=\frac{1}{\delta}(\max(-\beta(\dth-\vth)-r),\lambda)_{\ito}+(d_0,\lambda_0^+) \qquad \forall \lambda\in\Xth.
\end{equation}
We will split the error $\edh$ as follows
\[\edh=\dt-\dth=\dt-\dthtilde + \dthtilde -\dth\]
where $\dthtilde\in\Xth$ solves the auxiliary problem
\begin{equation}
\label{eq:auxiliaryproblemspatialerrord}
\bb(\dthtilde,\lambda)=\frac{1}{\delta}(\max(-\beta(\dthtilde-\vt)-r),\lambda)_{\ito}+(d_0,\lambda_0^+) \qquad \forall \lambda\in\Xth.
\end{equation}
The unique existence of $\dthtilde\in\Xth$ can be proven along the lines of the proof of Theorem \ref{thm:existencediscretestatenonlinear}. We start with an estimate for the second term.
\begin{lemma} \label{lem:firstparterrordth}
Let $\dth\in\Xth$ be the solution of \eqref{eq:reduceddiscreteODE} and let $\dthtilde\in\Xth$ be the solution of \eqref{eq:auxiliaryproblemspatialerrord} for a given right-hand side $l\in\Lzweizwei$ and an initial state $d_0\in\Lzwo$. Then, we have the error estimate
\begin{equation}
\label{eq:firsterrorestimatedth}
\Vert \dthtilde-\dth\Vert_{L^\infty(\I;\Lzwo)}\leq C\{h^2\Vert\nabla^2\vt\Vert_{L^1(\I;\Lzwo)}+\Vert \dt-\dthtilde\Vert_{\Leinszwei}\}
\end{equation}
with a constant $C>0$ independent of $\tau$ and $h$. 
\end{lemma}
\begin{proof}
A reduction of \eqref{eq:reduceddiscreteODE} and \eqref{eq:auxiliaryproblemspatialerrord} onto one subinterval yields
\begin{equation}
(\dthm,\lambda)=(d_{\tau h,m-1},\lambda)+\frac{\tau_m}{\delta}(\max(-\beta(\dthm-\vthm)-r),\lambda) \qquad \forall \lambda\in \Xh
\end{equation}
and
\begin{equation}
(\dthtildem,\lambda)=(\tilde{d}_{\tau h,m-1},\lambda)+\frac{\tau_m}{\delta}(\max(-\beta(\dthtildem-\varphi_{\tau,m})-r),\lambda) \qquad \forall \lambda\in \Xh .
\end{equation}
We subtract both equations, test with $\lambda=\dthm-\dthtildem$, apply Cauchy's inequality and the Lipschitz continuity of $\max$ to arrive at
\begin{align*}
\Vert \dthm-\dthtildem\Vert&\leq \Vert d_{\tau h,m-1}-\tilde{d}_{\tau h,m-1}\Vert + \frac{\beta}{\delta}\tau_m (\Vert\dthm-\dthtildem\Vert + \Vert\varphi_{\tau,m}-\vthm\Vert).
\end{align*}
We have seen in the proof of Lemma \ref{lem:preliminaryerrorvarphispatial} that
\[\Vert\varphi_{\tau,m}-\vthm\Vert \leq Ch^2\Vert\nabla^2\varphi_{\tau,m}\Vert +L_\Phi\Vert d_{\tau,m}-\dthm\Vert.\]
We insert this bound into to above estimation to obtain
\begin{align*}
&\Vert \dthm-\dthtildem\Vert \\
&\leq \frac{1}{1-\frac{\beta}{\delta}\tau_m(1+L_\Phi)} \left(\Vert d_{\tau h,m-1}-\tilde{d}_{\tau h,m-1}\Vert+\frac{\beta}{\delta}\tau_m(Ch^2\Vert\nabla^2\varphi_{\tau,m}\Vert + L_\Phi\Vert d_{\tau,m}-\dthtildem\Vert)\right). 
\end{align*}
Then, induction together with $d_{\tau h,0}=\tilde{d}_{\tau h,0}=P_hd_0$ yields
\begin{align*}
\Vert \dthm-\dthtildem\Vert\leq \sum\limits_{j=1}^m\prod\limits_{i=j}^m\frac{1}{1-\frac{\beta}{\delta}\tau_i(1+L_\Phi)}\left(Ch^2\tau_j\Vert\nabla^2\varphi_{\tau,j}\Vert+ C\tau_j\Vert d_{\tau,j}-\tilde{d}_{\tau h,j}\Vert\right).
\end{align*}
With the same arguments used in the proof of Lemma \ref{lem:stabestimateinftyzwo}, we may conclude
\[\Vert\dthm-\dthtildem\Vert \leq Ch^2\Vert\nabla^2\vt\Vert_{\Leinszwei}+C\Vert \dt-\dthtilde\Vert_{\Leinszwei}.\]
\end{proof}
For the estimation of the first term $\dt-\dthtilde$, we will make use of the following (reduced) auxiliary dual problem: Find a state $\pt\in\Xtau$ such that
\begin{equation} \label{eq:reduceddualproblem}
\bb(\lambda,\pt)+(f\pt,\lambda)_{\ito}=(\dt-\dthtilde,\lambda)_{\ito}
\end{equation}
is fulfilled for all $\lambda\in\Xtau$. The function $f$ should belong to $L^\infty(I\times\Omega)$ and $\vert f(t,x)\vert\leq\frac{\beta}{\delta}$ is assumed. The structure of the (reduced) dual problem is similar to the dual problem investigated in Proposition \ref{cor: stabestimatedualtimenonlinear} and Corollary \ref{cor:stabestimatesdualnonlinear}. Thus, existence of a unique solution $\pt\in\Xtau$ and stability estimates may be derived with the same arguments used to prove Proposition \ref{cor: stabestimatedualtimenonlinear} and Corollary \ref{cor:stabestimatesdualnonlinear}. We state the result in the following lemma:
\begin{lemma} \label{lem:existencestabilityreduceddual}
Let $f\in L^\infty(\I\times\Omega)$ with $\vert f(t,x)\vert \leq \frac{\beta}{\delta}$ be given. Then, the reduced dual problem \eqref{eq:reduceddualproblem} possesses a unique solution $\pt\in\Xtau$
provided that $\tau$ is sufficiently small. Moreover, the solution fulfills
\begin{equation}
\label{eq:stabestimatesreduceddual}
\Vert\pt\Vert^2_{\ito}+\sum\limits_{m=1}^M\tau_m^{-1}\Vert[\pt]_m\Vert^2 \leq C\Vert\dt-\dthtilde\Vert^2_{\ito}
\end{equation}
with a constant $C>0$ independent of $\tau$ and $h$. The jump term $[\pt]_M$ is defined as $-p_{\tau,M}^-$. 
\end{lemma}
We will further split the error $\dt-\dthtilde$ as 
\[\dt-\dthtilde=\underbrace{\dt-\pi^X_h \dt}_{=\etadh}+\underbrace{\pi^X_h \dt-\dthtilde}_{=\xidh}.\]
We need two auxiliary results for the error estimation: The first one is Galerkin orthogonality in space
\begin{equation}
\label{eq:Galerkinorthospacereduced}
\bb(\dt-\dthtilde,\lambda)=\frac{1}{\delta}(\max(-\beta(\dt-\vt)-r)-\max(-\beta(\dthtilde-\vt)-r),\lambda)_{\ito} \qquad \forall \lambda\in\Xh.
\end{equation}
The second one is boundedness of the approximation error $\xidh$:
\begin{lemma} \label{lem:approximationerrorxidhestimate}
For a sufficiently small $\tau$ we have
\begin{equation}
\label{eq:approximationerrorxidhestimateeq}
\Vert\xidh\Vert_{L^\infty\I;\Lzwo)}\leq C\Vert\etadh\Vert_{\Leinszwei}
\end{equation}
with a constant $C>0$ independent of $\tau$ and $h$. 
\end{lemma}
\begin{proof}
The idea of the proof is similar to the one of Lemma \ref{lem:firstparterrordth}. 
If we test $\xidhm$ with a test function $\lambda\in\Xh$, we have
\begin{align*}
(\xidhm,\lambda)&=(\xi^d_{h,m-1},\lambda)+\frac{\tau_m}{\delta}(\max(-\beta(d_{\tau,m}-\varphi_{\tau,m}-r)-\max(-\beta(\dthtildem-\varphi_{\tau,m})-r),\lambda)
\end{align*}
due to the definition of the projection $\pi_h$. We choose $\lambda=\xidhm$ and use Cauchy's inequality to obtain
\begin{align*}
\Vert\xidhm\Vert&\leq \Vert \xi^d_{h,m-1}\Vert+\frac{\beta}{\delta}\tau_m(\Vert\xidhm\Vert+\Vert\etadhm\Vert).
\end{align*}
Again, as $\xi^d_{h,0}=0$, induction leads to
\[\Vert \xidhm\Vert \leq \sum\limits_{j=1}^m\prod\limits_{i=j}^m\frac{1}{1-\frac{\beta}{\delta}\tau_i}\frac{\beta}{\delta}\tau_j\Vert\eta^d_{h,j}\Vert \leq C\Vert\etadh\Vert_{\Leinszwei}.\]
Note, that $1-\frac{\beta}{\delta}\tau_i>0$ for all $i=1,\ldots,M$ follows directly from $1-\frac{\beta}{\delta}\tau_i(1+L_\Phi)>0$ for all $i=1,\ldots,M$. 
\end{proof}

Before we proceed to the main result regarding the spatial error, we need to take a look at the spatial regularity of $\dt$:
\begin{theorem}
\label{thm:spatialregularityd}
Let $(\vt,\dt)\in\Vtau\times\Xtau$ be the solution of the semidiscrete state equation \eqref{eq:semidiscreteprimalproblemnonlinear}. Then, 
\[\dt\in L^\infty(\I;\Hs)\]
with $0\leq s<\frac{3}{2}$, provided that $d_0\in\Hs$. 
\end{theorem}
\begin{proof}
We have already proven the existence of a unique solution $\dt\in L^\infty(\I;\Lzwo)$ of the semidiscrete state equation \eqref{eq:semidiscreteprimalproblemnonlinear}. Thus, we only address the improved spatial regularity of $\dt$. We verify that $d_{\tau,m}\in\Hs$ for all $m=0,\ldots,M$. This will be achieved by induction. \\
For $m=0$, $d_{\tau,0}=d_0\in\Hs$ holds true by assumption. Thus, let  $d_{\tau,0},\ldots,d_{\tau,m-1}\in\Hs$. Then, $d_{\tau,m}\in\Lzwo$ solves \eqref{eq:semidiscretefixpointeqnonlinearoneinterval}
\[(d_{\tau,m},\lambda_m)=(d_{\tau,m-1},\lambda_m)+(\frac{1}{\delta}\max(-\beta(d_{\tau,m}-\varphi_{\tau,m})-r),\lambda_m)_{\imto}
\]
for all $\lambda_m\in\Lzwo$. We may rewrite this equation as
\[d_{\tau,m}(x)=d_{\tau,m-1}(x)+\frac{\tau_m}{\delta}\max(-\beta(d_{\tau,m}(x)-\varphi_{\tau,m}(x))-r) \]
which is fulfilled almost everywhere in $\Omega$. We will now express the $\max$ operator explicitly, that is, we distinguish between three cases:
\begin{enumerate}
\item Let $x\in \Omega^-_m:=\{x\in\Omega:-\beta(d_{\tau,m}(x)-\varphi_{\tau,m}(x))-r<0\}$. Then, the $\max$-operator equals zero and we have
\[d_{\tau,m}(x)=d_{\tau,m-1}(x).\]
\item Let $x\in \Omega^0_m:=\{x\in\Omega:-\beta(d_{\tau,m}(x)-\varphi_{\tau,m}(x))-r=0\}$. Then, the $\max$-operator vanishes as well and we have
\[d_{\tau,m}(x)=d_{\tau,m-1}(x)=\varphi_{\tau,m}(x)-\frac{r}{\beta}.\]
\item Let $x\in \Omega^+_m:=\{x\in\Omega:-\beta(d_{\tau,m}(x)-\varphi_{\tau,m}(x))-r>0\}$. In this case, we may express $d_{\tau,m}$ as
\begin{equation} \label{eq:updateformuladtmthirdcase}
d_{\tau,m}(x)=\frac{1}{1+\frac{\beta}{\delta}\tau_m}(d_{\tau,m-1}(x)+\frac{\beta}{\delta}\tau_m\varphi_{\tau,m}(x)-\frac{\tau_m}{\delta}r).
\end{equation}
\end{enumerate}
As the sets $\Omega^-_m,\Omega^0_m$ and $\Omega^+_m$ still depend on $d_{\tau,m}$, we have a look at the expression $-\beta(d_{\tau,m}(x)-\varphi_{\tau,m}(x))-r$ next. In the first two cases we may conclude, that if $-\beta(d_{\tau,m}(x)-\varphi_{\tau,m}(x))-r\leq 0$ then $-\beta(d_{\tau,m-1}(x)-\varphi_{\tau,m}(x))-r\leq 0$ has to hold true as well. In the third case, the insertion of \eqref{eq:updateformuladtmthirdcase} yields 
\[-\beta(d_{\tau,m}(x)-\varphi_{\tau,m}(x))-r=\frac{1}{1+\frac{\beta}{\delta}\tau_m}(-\beta(d_{\tau,m-1}(x)-\varphi_{\tau,m}(x))-r).\]
Thus, $-\beta(d_{\tau,m-1}(x)-\varphi_{\tau,m}(x))-r>0$ has to hold true as $\frac{1}{1+\frac{\beta}{\delta}\tau_m}$ is positive. The combination of these results shows that
\[-\beta(d_{\tau,m}(x)-\varphi_{\tau,m}(x))-r \leq 0 \Leftrightarrow -\beta(d_{\tau,m-1}(x)-\varphi_{\tau,m}(x))-r \leq 0 \]
and 
\[-\beta(d_{\tau,m}(x)-\varphi_{\tau,m}(x))-r > 0 \Leftrightarrow -\beta(d_{\tau,m-1}(x)-\varphi_{\tau,m}(x))-r  >0\]
by a contra-position argument. We may equivalently express $d_{\tau,m}(x)$ as
\[d_{\tau,m}(x)=\begin{cases} d_{\tau,m-1}(x) &, \omega(x)\leq 0, \\ \frac{1}{1+\frac{\beta}{\delta}\tau_m}(d_{\tau,m-1}(x)+\frac{\beta}{\delta}\tau_m\varphi_{\tau,m}(x))-\frac{\tau_m}{\delta}r)  &, \omega(x)>0, \end{cases}\]
with $\omega(x):=-\beta(d_{\tau,m-1}(x)-\varphi_{\tau,m}(x))-r$. \\
Next, we use $d_{\tau,m-1}(x)=\varphi_{\tau,m}(x)-\frac{1}{\beta}r-\frac{1}{\beta}\omega(x)$ to rewrite $d_{\tau,m}(x)$ again as
\begin{align*}
d_{\tau,m}(x)&=\begin{cases} \varphi_{\tau,m}(x)-\frac{1}{\beta}r-\frac{1}{\beta}\omega(x) &, \max(\omega(x)) = 0, \\ \varphi_{\tau,m}(x)-\frac{1}{\beta}r-\frac{1}{\beta}\frac{1}{1+\frac{\beta}{\delta}\tau_m}\omega(x)  &, \max(\omega(x))=\omega(x). \end{cases} \\
&=\varphi_{\tau,m}(x)-\frac{1}{\beta}r-\frac{1}{\beta}\omega(x)+\frac{1}{\beta}\max(\omega(x))\left(1-\frac{1}{1+\frac{\beta}{\delta}\tau_m}\right).
\end{align*}
Since $\varphi_{\tau,m}\in\Hzwei\hookrightarrow \Hs$ for $d_{\tau,m}\in\Lzwo$ and $s\leq 2$, $\omega\in\Hs$ due to the assumption $d_{\tau,m-1}\in\Hs$ and $\max(\omega)\in\Hs$ if and only if $0\leq s<\frac{3}{2}$ due to Lemma \ref{lem:propertiesmax}, we conclude $d_{\tau,m}\in\Hs$ for $0\leq s<\frac{3}{2}$ . 
\end{proof}
\begin{remark}
In general, $s=\frac{3}{2}-\epsilon$ is the maximal order of spatial differentiability for $\dt$. The restriction to $s<\frac{3}{2}$ goes back to the properties of the $\max$-operator as Runst and Sickel \cite{RS96} provided a counterexample that $\max(f)\not\in\Hs$ for $f\in\Hs$ and $s\geq\frac{3}{2}$. Therefore, the regularity is higher only in special cases, for example if $\omega(x)<0$ or $\omega(x)>0$ for all $x\in\Omega$ or if $\omega$ is sufficiently smooth in all $x\in\Omega$ with $\omega(x)=0$. 
\end{remark}
\begin{lemma} \label{lem:boundednessdtinHs}
For the solution of the semidiscrete state equation \eqref{eq:semidiscreteprimalproblemnonlinear} for a given right-hand side $l\in\Lzweizwei$ and a given initial state $d_0\in\Hs$, we have the boundedness 
\begin{equation}\label{eq:boundednessdtinHs}
\Vert \dt\Vert_{L^\infty(\I;\Hs)}\leq C\{\Vert d_0\Vert_{\Hs}+\Vert l\Vert_{L^1(\I;\Lzwo}\}
\end{equation}
with a constant $C>0$ independent of $\tau$, provided that $\tau$ is sufficiently small. 
\end{lemma}
\begin{proof}
We begin as in the proof of Lemma \ref{lem:stabestimateinftyzwo} by taking norms on both sides of \eqref{eq:semidiscretefixpointeqnonlinearoneinterval}. By making use of \eqref{eq: boundedmaxHs} we arrive at
\[\Vert \dtm\Vert_{\Hs}\leq \Vert d_{\tau,m-1}\Vert_{\Hs}+\frac{\beta}{\delta}\tau_m C(\Vert\dtm\Vert_{\Hs}+\Vert\vtm\Vert_{\Hzwei}).\]
Provided that $\frac{\beta}{\delta}C\tau_m<1$, we obtain
\[\Vert \dtm\Vert_{\Hs}\leq \frac{1}{1-\frac{\beta}{\delta}C\tau_m}\left(\Vert d_{\tau,m-1}\Vert_{\Hs}+\frac{\beta}{\delta}C\Vert\vt\Vert_{L^1(I_m;\Hzwei)}\right).\]
An argumentation similar to the proof of Lemma \ref{lem:stabestimateinftyzwo} yields
\[\Vert \dtm\Vert_{\Hs}\leq C\left(\Vert d_0\Vert_{\Hs}+ \Vert \vt\Vert_{L^1([0,t_m];\Hzwei)}\right).\]
The assertion follows from Lemma \ref{lem:stabestimateinftyzwo} and Theorem \ref{thm:stabestimatesprimalnonlinear}.
\end{proof}
We are now in position to prove our error estimates regarding the spatial error:
\begin{theorem} \label{thm:spatialerrornonlinear}
Let $l\in\Lzweizwei$ and $d_0\in\Hs, 0\leq s < \frac{3}{2}$ be given. For the errors $\evh=\vt-\vth$ and $\edh=\dt-\dth$ between the dG(0) semidiscretized solution $(\vt,\dt)\in\Vtau\times\Xtau$ of \eqref{eq:semidiscreteprimalproblemnonlinear} and the dG(0)cG(1) discretized solution $(\vth,\dth)\in\Vth\times\Xth$ of \eqref{eq:discretestateequationnonlinear}, we have the error estimate
\[\Vert\evh\Vert_{\ito}+\Vert\edh\Vert_{\ito}\leq Ch^s\{\Vert\nabla^2\vt\Vert_{\ito}+\Vert\dt\Vert_
{L^2(\I;\Hs)}\}\]
with a constant $C>0$ independent of $\tau$ and $h$.
\end{theorem}
\begin{proof} We only need to derive an estimate for $\dt-\dthtilde$ as all other terms have already been estimated. We consider the dual equation \eqref{eq:reduceddualproblem} with $f\in L^\infty(I\times\Omega)$ chosen as
\[f(t,x)=\begin{cases} \frac{\max(-\beta(\dthtilde(t,x)-\vt(t,x))-r)-\max(-\beta(\dt(t,x)-\vt(t,x))-r)}{\delta(\dt(t,x)-\dthtilde(t,x))} &,\text{ if } \dt(t,x)\not =\dthtilde(t,x) \\ 0 &, \text{ else. } \end{cases}\]
As $\max:\mathbb{R}\to\mathbb{R}$ is Lipschitz continuous, we have $\vert f(t,x)\vert\leq\frac{\beta}{\delta}$. We will denote the projection error for $\pt$ with $\etaph=\pt-\pi^X_h\pt$. Testing with $\lambda=\dt-\dthtilde$ yields
\begin{align*}
\Vert\dt-\dthtilde\Vert^2_{\ito}&=\bb(\dt-\dthtilde,\pt)+(f\pt,\dt-\dthtilde)_{\ito} \\
&=\bb(\dt-\dthtilde,\pi_h^X\pt)+\bb(\dt-\dthtilde,\etaph)+(f\pt,\dt-\dthtilde)_{\ito} \\
&=-(f(\dt-\dthtilde),\pi^X_h\pt)_{\ito}+(f\pt,\dt-\dthtilde)_{\ito}+\bb(\dt-\dthtilde,\etaph) \\
&=\bb(\xidh,\etaph)+\bb(\etadh,\etaph)+(f\xidh,\etaph)_{\ito}+(f\etadh,\etaph)_{\ito} .
\end{align*}
Here, we made use of Galerkin orthogonality in space \eqref{eq:Galerkinorthospacereduced}. We estimate the last four terms separately. For the first term, we have $\bb(\xidh,\etaph)=0$ due to the definition of $\pi^X_h$. The assertion follows directly if we insert $\xidh$ and $\etaph$ in the reduced bilinear form $\bb$. For the second term, we estimate with the dual representation of the reduced bilinear form
\begin{align*}
\bb(\etadh,\etaph) &=-\sum\limits_{m=1}^M([\pt]_m,\etadhm) \leq \left(\sum\limits_{m=1}^M\tau_m\Vert\etadhm\Vert^2\right)^\frac{1}{2}\left(\sum\limits_{m=1}^M\tau_m^{-1}\Vert[\pt]_m\Vert^2\right)^\frac{1}{2} \\
&\leq C\Vert\etadh\Vert_{\ito}\Vert \dt-\dthtilde\Vert_{\ito}.
\end{align*}
The last estimate follows from the dual stability \eqref{eq:stabestimatesreduceddual}. The third term may be estimated using Lemma \ref{lem:approximationerrorxidhestimate}.
For the last term, we directly have
\[(f\etadh,\etaph)_{\ito}\leq \frac{\beta}{\delta}\Vert\etadh\Vert_{\ito}\Vert\etaph\Vert_{\ito}.\]
Combing all estimates yields
\[\Vert\dt-\dthtilde\Vert^2_{\ito}\leq C\Vert\etadh\Vert_{\ito}(\Vert\etaph\Vert_{\ito}+\Vert \dt-\dthtilde\Vert_{\ito}). \]
As $\Vert\etaph\Vert_{\ito}\leq C\Vert\pt\Vert_{\ito}\leq C\Vert\dt-\dthtilde\Vert_{\ito}$ follows from the projection error estimate and the stability estimate for $\pt$, we obtain after division by $\Vert\dt-\dthtilde\Vert_{\ito}$ 
\[\Vert\dt-\dthtilde\Vert_{\ito} \leq C\Vert\etadh\Vert_{\ito}\leq Ch^s\Vert\dt\Vert_{L^2(\I;\Hs)}.\]
Lemma \ref{lem:firstparterrordth} then yields
\[\Vert\edh\Vert_{\ito}\leq Ch^2\Vert\nabla^2\vt\Vert_{\ito}+Ch^s\Vert\dt\Vert_{L^2(\I;\Hs)}\]
and finally Lemma \ref{lem:preliminaryerrorvarphispatial} gives us
\[\Vert\evh\Vert_{\ito}\leq Ch^2\Vert\nabla^2\vt\Vert_{\ito}+Ch^s\Vert\dt\Vert_{L^2(\I;\Hs)}.\]
Due to the reduced regularity of $\dt$, we only have $s<\frac{3}{2}$ such that the error in $\dt$ is dominant and the assertion follows. 
\end{proof}
Together, both error estimates yield the overall result
\begin{theorem} \label{thm:overallerrorestimatenonlinear}
Let $l\in\Heinszwei$ and $d_0\in\Hs$, $0\leq s<\frac{3}{2}$, be given. Let $(\varphi,d)\in V\times X$ be the solution of the continuous problem \eqref{eq:nonlinearweakformulationcont} and let $(\vth,\dth)\in\Vth\times\Xth$ be the solution of the dG(0)cG(1) discretized problem \eqref{eq:discretestateequationnonlinear}. Then, we have the error estimate
\begin{equation} \label{eq:overallerrorestimatenonlinear}
\Vert \varphi-\vth\Vert_{\ito}+\Vert d-\dth\Vert_{\ito}\leq C(\tau+h^s)
\end{equation} 
with a constant $C>0$ independent of $\tau$ and $h$. 
\end{theorem}
\begin{proof}
The assertion is a combination of the results from Theorems \ref{thm:temporalerrornonlinearstate} and \ref{thm:spatialerrornonlinear}. We only have to show the boundedness of $\Vert\partial_t\varphi\Vert_{\ito},\Vert\partial_t d\Vert_{\ito},\Vert \nabla^2\vt\Vert_{\ito}$ and $\Vert\dt\Vert_{L^2(\I;\Hs)}$ independent of $\tau$ and $h$. But this follows immediately from the stability estimates from Lemma \ref{lem:boundednessstatescontinuous}, Lemma \ref{lem:boundednessdtinHs} and Theorem \ref{thm:stabestimatesprimalnonlinear}.
\end{proof}
\section{Error estimates for the associated optimal control problem} \label{sec:errorestimatessemilinearoptimalcontrol}
This section is devoted to error estimation for the associated optimal control problem, that is, we want to measure the error between the continuous solution of the reduced optimal control problem
\[(P) \qquad \min j(l)=J(S(l),l), \quad \text{ s.t. } l\in\Heinsnullzwei \]
with $S:\LzweiHminuseins\to V\times X$, $S(l)=(\varphi,d)$, being the solution operator of problem \eqref{mod:nonlinmodbegin}-\eqref{mod:nonlinmodend} and the solution of a discrete version
\[(P_\sigma)\qquad  \min \jth(l)=\frac{1}{2}\Vert \vth-\varphi_d\Vert^2_{\ito}+\frac{1}{2}\Vert \dth-d_d\Vert^2_{\ito}+\frac{\alpha_l}{2}\Vert l\Vert^2_{L_\sigma}, \quad \text{ s.t. } l\in L_\sigma \]
with $\Sth:\LzweiHminuseins\to\Vth\times\Xth$, $\Sth(l)=(\vth,\dth)$, being the solution operator of the space-time discrete problem \eqref{eq:discretestateequationnonlinear}. The parameter $\sigma$ denotes the discretization parameters for the control in time and space. The discrete control space $L_\sigma$ will be chosen later. In particular, we will investigate a variational discretization as well as a nonconforming dG(0)cG(1) discretization of the control. For simplicity the control will be discretized with the same parameter as the state equation if it is discretized at all. Thus, we have $\sigma=(\tau,h)$ for discretized controls. Recall, that we employ the $\Heinszwei$-seminorm for the objective. The norm $\Vert\cdot\Vert_{L_\sigma}$ will be chosen appropriately. \\
  
\subsection{The continuous optimal control problem}
We begin with solvability of problem $(P)$. 
\begin{theorem} \label{thm:existenceofoptimalcontrolscontinuous}
Let $\varphi_d,d_d\in\Lzweizwei$ be two given desired states, let $d_0\in\Lzwo$ be a given initial state and let $\alpha_l>0$ hold true. Then, the optimal control problem $(P)$ admits at least one solution $\ovl\in\Heinsnullzwei$.
\end{theorem}
\begin{proof}
The assertion may be proven with standard arguments and relies on the compact embedding $\Heinsnullzwei\hookrightarrow\hookrightarrow \LzweiHminuseins$ as well as the Lipschitz continuity of $S:\LzweiHminuseins\mapsto V\times X$.
\end{proof}

The semilinear problem may possess multiple minimizers as our problem is not necessarily convex. Therefore, we are dealing with local minimizers.
\begin{definition} \label{def:localminimizer}
A control $\ovl\in\Heinsnullzwei$ is called a local solution of $(P)$ in the sense of $\Heinsnullzwei$ if there exists a constant $\rho>0$ such that
\[j(\ovl)\leq j(l)\]
is satisfied for all $l\in B_\rho(\ovl):=\{l\in\Heinsnullzwei : \Vert l-\ovl\Vert_{\Heinsnullzwei} < \rho\}$. It is called a strict local minimizer if the inequality is strict for $l\not =\ovl$. It is called a local solution of $(P)$ in the sense of $\Lzweizwei$ if the above inequality holds true for all $l\in B_\rho^I:=\{l\in\Heinsnullzwei : \Vert l-\ovl\Vert_{\ito} < \rho\}$.
\end{definition}
In the following $\rho>0$ always refers to the radius of local optimality of a local optimal control $\ovl$. To ensure that we are only dealing with strict local minimizers, we impose the following quadratic growth condition.
\begin{assumption} \label{ass:quadgrowthcond}
For a local minimizer $\ovl\in\Heinsnullzwei$ of $(P)$, there exists a constant $\gamma>0$ such that the quadratic growth condition
\begin{equation} \label{eq:quadgrowthcond}
\gamma\Vert l-\ovl\Vert^2_{\Lzweizwei}\leq j(l)-j(\ovl)
\end{equation}
is satisfied for all $l\in B_\rho(\ovl)$. 
\end{assumption}

Necessary optimality conditions are derived in detail for a slightly different problem in \cite{LB19}. The author chooses $\Lzweizwei$ as control space although existence of an optimal control for this setting is unknown (see the introduction) and not addressed in \cite{LB19}. In this contribution, strong stationarity conditions equivalent to purely primal necessary optimality conditions of the form 
\begin{equation} \label{eq:primaloptimality}
j^\prime(\ovl;\delta l)\geq 0 \qquad \forall \delta l\in\Heinsnullzwei
\end{equation}
are established. The proofs are adaptable to our case with $\Heinsnullzwei$ as control space. Therefore, we state the following strong stationarity conditions without proof. 
\begin{theorem}[see \cite{LB19}] \label{thm:optimalitysystemforP}
Let $\ovl\in\Heinsnullzwei$ be locally optimal for $(P)$ with associated states $(\ovv,\ovd)\in V\times X$. Then, there exist unique adjoint states $(\ovz,\ovp)\in V\times X$ with $p(T)=0$ and a unique multiplier $\mu\in\Lzweizwei$ which fulfill 
\begin{equation}
\label{eq:adjointeqP}
B((\psi,\lambda),(\ovz,\ovp))=\frac{\beta}{\delta}(\mu,\psi-\lambda)_{I\times\Omega}+(\ovv-\varphi_d,\psi)_{I\times\Omega}+(\ovd-d_d,\lambda)_{I\times\Omega}
\end{equation}
for all $(\psi,\lambda)\in V\times X$. Moreover, the variational equality
\begin{equation}
\label{eq:variationalinequalityP}
(\ovz,\delta l)_{\ito}+\alpha_l(\ovl,\delta l)_{\Heinsnullzwei} = 0 \qquad \forall \delta l\in\Heinsnullzwei
\end{equation}
is satisfied. Finally, 
\begin{equation} \label{eq:signconditionmu}
\mu(t,x)\begin{cases} =\chi_{\Omega^+_t}(t,x)\ovp(t,x) &,\text{  a.e. in } \Omega^+_t\cup \Omega^-_t \\
\in [0,\ovp(t,x)] &,\text{  a.e. in } \Omega^0_t
 \end{cases}
\end{equation}
with $\Omega^+_t,\Omega^-_t$ and $\Omega^0_t$ being the active, inactive and biactive set at time point $t$.
\end{theorem}
\begin{remark}
The equations \eqref{eq:adjointeqP} and \eqref{eq:variationalinequalityP} as well as the regularity of $\ovz,\ovp$ and $\mu$ can be derived via regularization of $\max$, for example with the choice $\maxe\colon\mathbb{R}\to\mathbb{R}$
\begin{equation} \label{ex:maxe}
\maxe(x)=\begin{cases} 0 &,x\leq 0, \\ -\frac{1}{2\varepsilon^3}x^4+\frac{1}{\varepsilon^2}x^3 &, x\in (0,\varepsilon) \\ x-\frac{\varepsilon}{2} &, x\geq\varepsilon. \end{cases} 
\end{equation}
It is the sign condition \eqref{eq:signconditionmu} that ensures the equivalence to the primal optimality condition \eqref{eq:primaloptimality}. Without it, the system is only weakly stationary.
\end{remark}
\begin{remark}
The variational formulation \eqref{eq:variationalinequalityP} is the weak formulation of the second order ODE
\[-\partial_{tt}l=\frac{1}{\alpha_l}z\]
with boundary conditions $l(0)=0, \partial_t l(T)=0$.
Since $z\in V$, we may conclude that $\partial_{tt}\ovl\in V$ which implies $\ovl\in V$.
\end{remark}

\subsection{A priori error estimates for the optimal control} \label{ch:errorestimatescontrolnonlinear}
In this section, we finally prove the convergence of two discretization techniques for the control. The quadratic growth condition from Assumption \ref{ass:quadgrowthcond} then yields first error estimates for these discretization techniques. Note, that these estimates are not optimal as we will see in the numerics. 
\subsubsection*{Variational discretization}
We begin with a variational discretization of $(P)$, that is we choose $L_\sigma=\Heinsnullzwei$ in $(P_\sigma)$. To express that the control is not discretized, we refer to the variationally discretized problem as $(P_{\tau h})$
\[(P_{\tau h}) \quad \min \jth(l)=\frac{1}{2}\Vert \vth-\varphi_d\Vert^2_{\ito}+\frac{1}{2}\Vert \dth-d_d\Vert^2_{\ito}+\frac{\alpha_l}{2}\Vert l\Vert^2_{\Heinsnullzwei} \]
with $l\in L_\sigma=\Heinsnullzwei$. The existence of a solution $\ovlth$ may be proven with the exact same arguments used for the proof of Theorem \ref{thm:existenceofoptimalcontrolscontinuous} as $S$ and $\Sth$ are both weakly continuous for arguments in $\Heinsnullzwei$. \\
\begin{theorem} \label{thm:convergencevariational}
Let $\ovl\in\Heinsnullzwei$ be a local solution of $(P)$ and let Assumption \ref{ass:quadgrowthcond} be satisfied for $\ovl$. For every $\sigma:=(\tau,h)$, there exists a local solution $\ovlth\in\Heinsnullzwei$ of $(P_{\tau h})$ such that $\ovlth \to \ovl$ in $\Heinsnullzwei$ for $\sigma\to (0,0)$.
\end{theorem}
\begin{proof} The proof is standard but we will give a sketch for the convenience of the reader. We denote the reduced discrete objective by $\jth(l):=J(\Sth(l),l)$ and define the auxiliary optimal control problem
\[(P_{\tau h}^\rho)\begin{cases} \min \jth(l) \\
s.t. \,\; l\in B_\rho(\ovl) \end{cases}
\]
with $\rho>0$ being the radius of strict local optimality of $\ovl$. This problem admits global solutions whose existence may be proven by the same standard arguments used before. Let $\{\ovlth\}$ be a sequence of global minimizers of $(P_{\tau h}^\rho)$. Since $\{\ovlth\}\subset B_\rho(\ovl)$, there exists $\tilde{l}\in B_\rho(\ovl)$ and a weakly converging subsequence with $\ovlth\rightharpoonup \tilde{l}$ in $\Heinsnullzwei$. Due to compact embedding, this convergence is strong in $\LzweiHminuseins$.  As a consequence of the Lipschitz continuity of $\Sth$ and the error estimates proven in the last section, we have the convergence of the states in $\Lzweizwei\times\Lzweizwei$ that is
\begin{align}
\Vert \Sth(\ovlth)-S(\tilde{l})\Vert_{(\ito)^2} \to 0
\end{align}
holds true for $\sigma\to(0,0)$. Note, that the error estimate is valid only due to $\tilde{l}\in\Heinsnullzwei$. As $\Vert \cdot\Vert_{\Heinsnullzwei}$ is weakly lower semi-continuous, we obtain 
\begin{equation}
\label{eq:convergerncevariationmalhilfsresultat2}
j(\tilde{l})\leq \liminf_{\sigma\to (0,0)} \jth(\ovlth)\leq \limsup_{\sigma\to(0,0)} \jth(\ovlth)\leq \limsup_{\sigma\to(0,0)}\jth(\ovl) = j(\ovl) 
\end{equation}
The last inequality follows from admissibleness of $\ovl$ for $(P_{\tau h}^\rho)$ and the convergence is a direct consequence of the error estimates from the last section for $\ovl\in\Heinsnullzwei$. \\
The quadratic growth condition from Assumption \ref{ass:quadgrowthcond} ensures that $\ovl$ is a strict local minimizer of $(P)$. Thus, $\tilde{l}\in B_\rho(\ovl)$ together with \eqref{eq:convergerncevariationmalhilfsresultat2} directly leads to $\ovl=\tilde{l}$, that is $\ovlth\rightharpoonup \ovl$ in $\Heinsnullzwei$. Moreover, we have shown that $\jth(\ovlth)\to j(\ovl)$. Together with the strong convergence of the states, this immediately yields norm convergence of the controls and thus strong convergence of the controls $\ovlth\to\ovl$ in $\Heinsnullzwei$. It remains to prove that $\ovlth$ is indeed a local minimizer of $(P_{\tau h})$. But this follows from the fact that $\Vert \ovl-\ovlth\Vert_{\Heinsnullzwei}<\rho$ for $\sigma$ small enough. 
\end{proof}

Mere convergence of $\ovlth\to\ovl$ enables us to apply the quadratic growth condition to the pair $\ovlth,\ovl$. In this way, we obtain an a priori error estimate for the variational discretization. 
\begin{theorem} \label{thm:convergencevariationalerrorestimate}
Let $\ovl\in\Heinsnullzwei$ be a local minimizer for $(P)$, let Assumption \ref{ass:quadgrowthcond} be satisfied for $\ovl$ and let $\{\ovlth\}$ be a sequence of local minimizers of $(P_{\tau h})$ which converges to $\ovl$. Then, we have the error estimate
\begin{equation}
\label{eq:errorestimatevariational}
\Vert \ovl-\ovlth\Vert_{\Lzweizwei}\leq C\left(\tau^\frac{1}{2}+h^\frac{s}{2}\right)
\end{equation}
with $1\leq s<\frac{3}{2}$ and $C>0$ independent of $\tau$ and $h$. 
\end{theorem}
\begin{proof}
The essential ingredient for the proof is the quadratic growth condition. As $\ovlth\to\ovl$ in \linebreak $\Heinsnullzwei$, we have $\Vert\ovlth-\ovl\Vert_{\Heinsnullzwei}<\rho$ for $(\tau,h)$ small enough. Thus, we may estimate 
\begin{align*}
\gamma \Vert\ovlth-\ovl\Vert_{\Lzweizwei}^2 & \leq j(\ovlth)-j(\ovl) \\
&= (j(\ovlth)-\jth(\ovlth)) +(\jth(\ovlth)-\jth(\ovl)) + (\jth(\ovl)-j(\ovl)) \\
&= (I) + (II) + (III)
\end{align*}
We will estimate the three terms separately and start with term $(II)$. As $\ovlth$ is a global minimizer of $(P_{\tau h})$ in $B_\rho(\ovl)$ and $\ovl$ is admissible for $(P_{\tau h})$, we have $\jth(\ovlth)\leq \jth(\ovl)$ which is equivalent to $(II)\leq 0$. Thus, this term can be omitted for the error estimation. \\
The first and the third term are finite element errors for the state equation. Both can be estimated in the same way. We will present the error estimate for the first term in detail. 
\begin{align*}
j(\ovlth)-\jth(\ovlth)&=\frac{1}{2}\left(\Vert S(\ovlth)-y_d\Vert_{(\ito)^2}^2-\Vert \Sth(\ovlth)-y_d\Vert_{(\ito)^2}^2\right) \\
&=\frac{1}{2}(S(\ovlth)-\Sth(\ovlth), S(\ovlth)+\Sth(\ovlth)-2y_d)_{(\ito)^2} \\
&\leq \frac{1}{2}(\Vert S(\ovlth)\Vert_{(\ito)^2}+ \Vert \Sth(\ovlth)\Vert_{(\ito)^2} + 2\Vert y_d\Vert_{(\ito)^2})\Vert S(\ovlth)-\Sth(\ovlth)\Vert_{(\ito)^2}
\end{align*}
Boundedness of the states both for the continuous (see Lemma \ref{lem:boundednessstatescontinuous}) and the discrete (see Theorem \ref{thm:stabestimatesprimalnonlinear}) states together with the boundedness  of $\Vert \ovlth\Vert_{\ito}$ independent of $\tau$ and $h$ yields boundedness of the first term. For the second term, we obtain due to the error estimate from Theorem \ref{thm:overallerrorestimatenonlinear} 
\[\Vert S(\ovlth)-\Sth(\ovlth)\Vert_{(\ito)^2}\leq C(\tau+h^s).\]
Taking square roots on both sides then yields the assertion. 
\end{proof}

\subsubsection*{dG(0)cG(1) discretization}
A more interesting discretization technique is the choice $L_\sigma=\Vth$. Since $\Vth\not\subset\Heinsnullzwei$, we have to modify the discrete objective and choose $\Vert l\Vert_{L_\sigma}^2=\sum\limits_{m=1}^M\Vert\partial_t l\Vert_{\imto}^2+\sum\limits_{m=1}^M\tau_m^{-1}\Vert [l]_{m-1}\Vert^2$.
Thus, in this subsection we prove the convergence of local minimizers $\ovls\in\Vth$ of
\[(P_{\sigma}) \min \hat{J}(\vth,\dth,l)=\frac{1}{2}\Vert \vth-\varphi_d\Vert^2_{\ito}+\frac{1}{2}\Vert \dth-d_d\Vert^2_{\ito}+\frac{\alpha_l}{2}\Vert l\Vert^2_{L_\sigma}\]
towards a strict local minimizer $\ovl\in\Heinsnullzwei$ of $(P)$. The idea of the proof is still the same as before but one has to be more careful since expressions such as $j(\ovls)$ are not necessarily well defined. This fact has a direct consequence for the localization argument because we cannot work with $B_\rho(\ovl)\cap\Vth$ as the set of admissible controls in the auxiliary problem. Instead, we choose $B^I_\rho(\ovl)\cap\Vth$.  

\begin{theorem} \label{thm:convergencecgcg}
Let $\ovl\in\Heinsnullzwei$ be a strict local solution of $(P)$ in the sense of $\Lzweizwei$. For every $\sigma:=(\tau,h)$, there exists a local solution $\ovls\in\V{0}{1}$ of $(P_{\sigma})$ such that $\ovls \to \ovl$ in $\Lzweizwei$ for $\sigma\to (0,0)$. 
\end{theorem}
\begin{proof} We reduce the function $\hat{J}$ onto the control via $\jth(l):=\hat{J}(\Sth(l),l)$. Then, we define the auxiliary problem
\[(P_{\sigma}^\rho)\begin{cases} \min \jth(l) \\
s.t. \quad l\in B^I_\rho(\ovl)\cap\V{0}{1} \end{cases}
\]
Again, there exists a global solution $\ovls$ for each $\sigma$. The sequence $\{\ovls\}$ is now bounded in $\Lzweizwei$, i.e. there exists $\tilde{l}\in B^I_\rho(\ovl)$ such that $\ovl_\sigma\rightharpoonup \tilde{l}$. Unfortunately, we cannot conclude the strong convergence of the associated states as we are lacking a compact embedding. We start with an estimate from above for $\jth(\ovls)$. We require a projection of $\ovl$ onto $\Vth$ since $\ovl$ is not admissible. We choose $\Pi\ovl:=\pi_h^V(I_\tau\ovl)$ which converges strongly to $\ovl$ in $\Lzweizwei$. Thus, we have $\Sth(\Pi\ovl)\to S(\ovl)$ in $\Lzweizwei\times\Lzweizwei$ as well as the admissibleness of $\Pi\ovl$ for $\sigma$ small enough. Moreover, as $P_h^V$ is Lipschitz continuous with constant 1 we obtain
\begin{align*}
\sum\limits_{m=1}^M\tau_m^{-1}\Vert[\Pi\ovl]_{m-1}\Vert^2& \leq \sum\limits_{m=1}^M\tau_m^{-1}\Vert I_\tau\ovl(t_m)-I_\tau\ovl(t_{m-1})\Vert^2\\
&=\sum\limits_{m=1}^M\tau_m^{-1}\Vert \ovl(t_m)-\ovl(t_{m-1})\Vert^2  \rightarrow \Vert\partial_t l\Vert_{\ito}^2. 
\end{align*}

Finally, we arrive at
\begin{align}
\jth(\ovls)&\leq \jth(\Pi\ovl) \leq \frac{1}{2}\Vert \Sth(\Pi\ovl)-y_d\Vert^2_{(\ito)^2}+ \frac{\alpha_l}{2}\sum\limits_{m=1}^M\Vert \frac{\ovl(t_m)-\ovl(t_{m-1})}{\tau_m}\Vert^2_{\imto} \nonumber\\
 \label{eq:convergencedgcghilfsresultat1}
&\rightarrow j(\ovl).
\end{align}
In particular, this upper bound for $\jth(\ovls)$ also provides an upper bound  \begin{equation} \label{eq:convergencedgcghilfsresultat2}
\sum\limits_{m=1}^M\tau_m^{-1}\Vert [\ovls]_{m-1}\Vert^2\leq C
\end{equation}
with $C>0$ independent of $\tau$ and $h$.  \\
Next, we will provide a lower bound for $\liminf_{\sigma\to(0,0)}\jth(\ovls)$. To overcome the current nonexistence of $j(\tilde{l})$ due to low temporal regularity we will prove higher temporal regularity of $\tilde{l}$ first. We define an approximation $\hat{l}_\sigma$ of $\ovl_\sigma$ which belongs to $\Heinsnullzwei$ via
\[\hat{l}_\sigma(t):=\begin{cases} \frac{1}{\tau_1}t\ovls(t_1) &,\text{ if } t\in I_1, \\
\ovl_{\sigma,m-1}+\frac{1}{\tau_m}(t-t_{m-1})[\ovl_\sigma]_{m-1} &,\text{ if } t\in I_m,m=2,\ldots,M. \end{cases}\]
For $\hat{l}_\sigma$, we have the following properties:
\begin{itemize}
\item[(i)] $\Vert \partial_t\hat{l}_\sigma\Vert_{\ito}^2=\sum\limits_{m=1}^M \tau_m\Vert\frac{1}{\tau_m}[\ovl_\sigma]_{m-1}\Vert^2=\sum\limits_{m=1}^M\tau_m^{-1}\Vert[\ovl_\sigma]_{m-1}\Vert^2$
\item[(ii)] $\Vert \hat{l}_\sigma-\ovl_\sigma\Vert_{\ito}^2=\frac{1}{3}\sum\limits_{m=1}^M\tau_m\Vert[\ovl_\sigma]_{m-1}\Vert^2$
\end{itemize}
Due to \eqref{eq:convergencedgcghilfsresultat2}, (i) indicates that $\Vert \hat{l}_\sigma\Vert_{\Heinsnullzwei}\leq C$. Hence, there exists a subsequence converging weakly in $\Heinsnullzwei$ towards a function $g\in\Heinsnullzwei$. This convergence is again strong in $\LzweiHminuseins$ which implies the convergence of the states. Equation \eqref{eq:convergencedgcghilfsresultat2} in combination with equation (ii) yields
\begin{equation} \label{eq:convergencedgdgHilfsresultat4}
\Vert\hat{l}_\sigma-\ovls\Vert_{\ito}^2\leq \frac{1}{3}\tau^2\sum\limits_{m=1}^M\tau_m^{-1}\Vert[\ovl_\sigma]_{m-1}\Vert^2\leq C\tau^2\to 0.
\end{equation}
Thus, we may conclude 
\begin{itemize}
\item[(iii)] $\Vert\Sth(\hat{l}_\sigma)-\Sth(\ovls)\Vert_{(\ito)^2}\leq L_S\Vert\hat{l}_\sigma-\ovls\Vert_{\ito} \to 0$
\end{itemize}
as well as
\begin{itemize}
\item[(iv)] $\Vert \ovls-g\Vert_{\LzweiHminuseins}\leq \Vert \hat{l}_\sigma-\ovls\Vert_{\ito} + \Vert \hat{l}_\sigma-g\Vert_{\LzweiHminuseins} \to 0.$
\end{itemize}
Because of the weak convergence $\ovls\rightharpoonup \tilde{l}$ in $\Lzweizwei$ which also holds true in $\LzweiHminuseins$, we have $\tilde{l}=g$ as a consequence of the uniqueness of weak limits. Since $g\in\Heinsnullzwei$, we finally showed that $\tilde{l}$ possesses enough temporal regularity for expressions $j(\tilde{l})$ to be well defined. 
From here, we conclude
\begin{align*}
j(\tilde{l})=j(g)&\leq \liminf_{\sigma\to(0,0)} \jth(\hat{l}_\sigma) \leq \limsup_{\sigma\to(0,0)} \jth(\hat{l}_\sigma) \\
&= \limsup_{\sigma\to(0,0)} \frac{1}{2}\Vert \Sth(\hat{l}_\sigma)-y_d\Vert_{(\ito)^2}^2+\frac{\alpha_l}{2}\Vert\partial_t\hat{l}_\sigma\Vert^2_{\ito} \\
&\underset{(\star)}{=}\limsup_{\sigma\to(0,0)} \frac{1}{2}\Vert \Sth(\ovls)-y_d\Vert_{(\ito)^2}^2+\frac{\alpha_l}{2}\sum\limits_{m=1}^M\tau_m^{-1}\Vert [\ovls]_{m-1}\Vert^2 \\
&=\limsup_{\sigma\to(0,0)} \jth(\ovls) \leq \limsup_{\sigma\to(0,0)} \jth(\Pi\ovl) =j(\ovl).
\end{align*}
Note, that equality $(\star)$ is due to (i) and (iii). Since there holds $\tilde{l}\in B^I_\rho(\ovl)$, we arrive at $\tilde{l}=g=\ovl$. The convergence $\jth(\hat{l}_\sigma)\to j(\ovl)$ together with the convergence of the states $\Sth(\hat{l}_\sigma)\to S(\ovl)$ then yields the strong convergence of $\hat{l}_\sigma\to\ovl$ in $\Heinsnullzwei$. In particular, we have the strong convergence of $\partial_t\hat{l}_\sigma$ to $\partial_t\ovl$ in $\Lzweizwei$ which is nothing else than the convergence of $\sum\limits_{m=1}^M\tau_m^{-1}\Vert[\ovls]_{m-1}\Vert^2$ to $\Vert\partial_t\ovl\Vert_{\ito}^2$. Together with (iii), we obtain the convergence $\jth(\ovls)\to j(\ovl)$. Finally, \eqref{eq:convergencedgdgHilfsresultat4} yields
\[\Vert \ovls-\ovl\Vert_{\ito}\leq \Vert \ovls-\hat{l}_\sigma\Vert_{\ito}+\Vert \hat{l}_\sigma-\ovl\Vert_{\Heinsnullzwei}\to 0.\]
Local optimality of $\ovls$ for $(P_\sigma)$ now follows again immediately from this convergence. 
\end{proof}

Since the quadratic growth condition is not applicable for the dG(0)cG(1)-discretization, we will again make use of $\hat{l}_\sigma$ to derive an error estimate for $\ovls-\ovl$. 
\begin{theorem} \label{thm:convergencedgcgerrorestimate}
Let $\ovl\in\Heinsnullzwei$ be a strict local minimizer for $(P)$ in the sense of \linebreak $\Lzweizwei$, let Assumption \ref{ass:quadgrowthcond} be satisfied and let $\{\ovls\}$ be a sequence of local minimizers of $(P_{\sigma})$ which converges to $\ovl$ in $\Lzweizwei$. Then, we have the error estimate
\begin{equation}
\label{eq:errorestimatedgcg}
\Vert \ovl-\ovls\Vert_{\Lzweizwei}\leq C\left(\tau^\frac{1}{2}+h^\frac{s}{2}\right)
\end{equation}
with $s<\frac{3}{2}$ and $C>0$ independent of $\tau$ and $h$. 
\end{theorem}
\begin{proof}
A direct application of the strategy used in detail in the proof of Theorem \ref{thm:convergencevariationalerrorestimate} for $\ovls-\ovl$ is not possible as Theorem \ref{thm:overallerrorestimatenonlinear} is not applicable to the error $\Vert \Sth(\ovls)- S(\ovls)\Vert_{(\ito)^2}$ due to low temporal regularity of $\ovls$. Thus, instead, we consider $\hat{l}_\sigma\in\Heinsnullzwei$ as defined in the last proof. We have proved that $\hat{l}_\sigma\to\ovl$ in $\Heinsnullzwei$. Therefore, for $\hat{l}_\sigma$ the quadratic growth condition is applicable and yields 
\[\Vert\hat{l}_\sigma-\ovl\Vert_{\Lzweizwei}\leq C(\tau^\frac{1}{2}+h^\frac{s}{2})\]
in the same way as shown before. Then, due to \eqref{eq:convergencedgdgHilfsresultat4} we conclude
\[\Vert \ovls-\ovl\Vert_{\ito}\leq \Vert \ovls-\hat{l}_\sigma\Vert_{\ito}+\Vert\hat{l}_\sigma-\ovl\Vert_{\Lzweizwei} \leq C\tau+C(\tau^\frac{1}{2}+h^\frac{s}{2}).\]
\end{proof}

\section{Numerical examples}
\subsection{Simulation}
For the simulation, after discretization we employ a fixed point argument to solve the discrete nonlinear system of equations. This implies, that we have to choose $\tau$ small enough based on our findings about the existence of solutions. 

We will illustrate the discretization error estimates in two steps. First, we refine the temporal discretization parameter $\tau$ while the spatial discretization parameter $h$ will be fixed. In a second experiment, the roles will be switched. For simplicity, we employ equidistant meshes in both space and time. We use the abbreviations $e^\varphi_{\tau h}:= \varphi-\vth$ and $e^d_{\tau h}:=d-\dth$.\\

In the first example, the biactive set is of zero measure in each time point and moving in time. Therefore, the local damage $d$ has kinks in space which are moving in time. We consider the one-dimensional domain $\Omega=(0,1)$ for this example and set $T=1$. Consider
\[\varphi_1(t,x)=\sin(3\pi x)t,\]
\[d_1(t,x)=\begin{cases} d_0(x) &, t\leq t_a(x) \text{ or } \varphi_1(t,x)\leq 0, \\ \sin(3\pi x)t-\frac{r}{\beta}\\
-\frac{\delta}{\beta}\sin(3\pi x)[1-\exp(\frac{\beta}{\delta}(t_a(x)-t))] &, t\geq t_a(x) \text{ and } \varphi_1(t,x)>0 \end{cases}\]
with $\alpha=1$, $\beta=50$, $\delta=0.1$, $r=0.25\beta$ and $d_0(x)=0$.
For every point in space $x\in\Omega$, $t_a(x)$ is the point in time at which $x$ becomes active. It is given as
\[t_a(x)=\frac{r}{\beta\sin(3\pi x)}.\]
The load is given as
\[l_1(t,x)=(9\alpha\pi^2+\beta)\varphi_1(t,x)-\beta d_1(t,x).\]

Table \ref{tab:errorstatefirst} depicts the simulation results. We observe that the error in $\varphi$ converges faster than predicted by the theory. The rate for the error in $d$ is clearly smaller than the rate for $\varphi$, although, with $1.69$ on average, it is mildly larger than the predicted rate of $1.5$.
{\small
\begin{table}[h] \centering
\begin{tabular}{|c|c c|c c||c|c c|c c|}
\hline
$h$&\multicolumn{4}{c||}{$2^{-9}$}&$\tau$&\multicolumn{4}{c|}{$2^{-9}$} \\
\hline
$\tau$& $\Vert e^\varphi_{\tau h}\Vert_{\ito}$ & EOC & $\Vert e^d_{\tau h}\Vert_{\ito}$ & EOC & $h$ & $\Vert e^\varphi_{\tau h}\Vert_{\ito}$ & EOC & $\Vert e^d_{\tau h}\Vert_{\ito}$ & EOC \\
\hline
$2^{-9}$ & 7.61e-04 & -    & 5.11e-04 & - & $2^{-3}$ & 9.53e-02 & -    & 8.62e-02 & - \\
$2^{-10}$ & 3.62e-04 & 1.06 & 2.52e-04 & 1.02 & $2^{-4}$ &2.73e-02 & 1.80 & 2.81e-02 &1.61\\
$2^{-11}$ & 1.64e-04 & 1.14 & 1.31e-04 & 0.93 &$2^{-5}$ & 6.96e-03 & 1.97 & 8.02e-03 & 1.81\\
$2^{-12}$ & 6.74e-05 & 1.28 & 8.57e-05 & 0.62 &$2^{-6}$ & 1.70e-03 & 2.03 & 2.35e-03 &1.77\\
$2^{-13}$ &2.80e-05 & 1.26 & 7.88e-05 & 0.12 &$2^{-7}$ & 4.30e-04 &1.98 & 7.46e-04 &1.65\\
 &  &  &  &  &$2^{-8}$ & 1.08e-04 & 1.98 &2.41e-04 &1.62\\
\hline
\end{tabular}
\caption{1st example: Errors for the states $\varphi_1, d_1$}
\label{tab:errorstatefirst}
\end{table}}

\begin{figure} [h] \centering
\includegraphics[width=\textwidth]{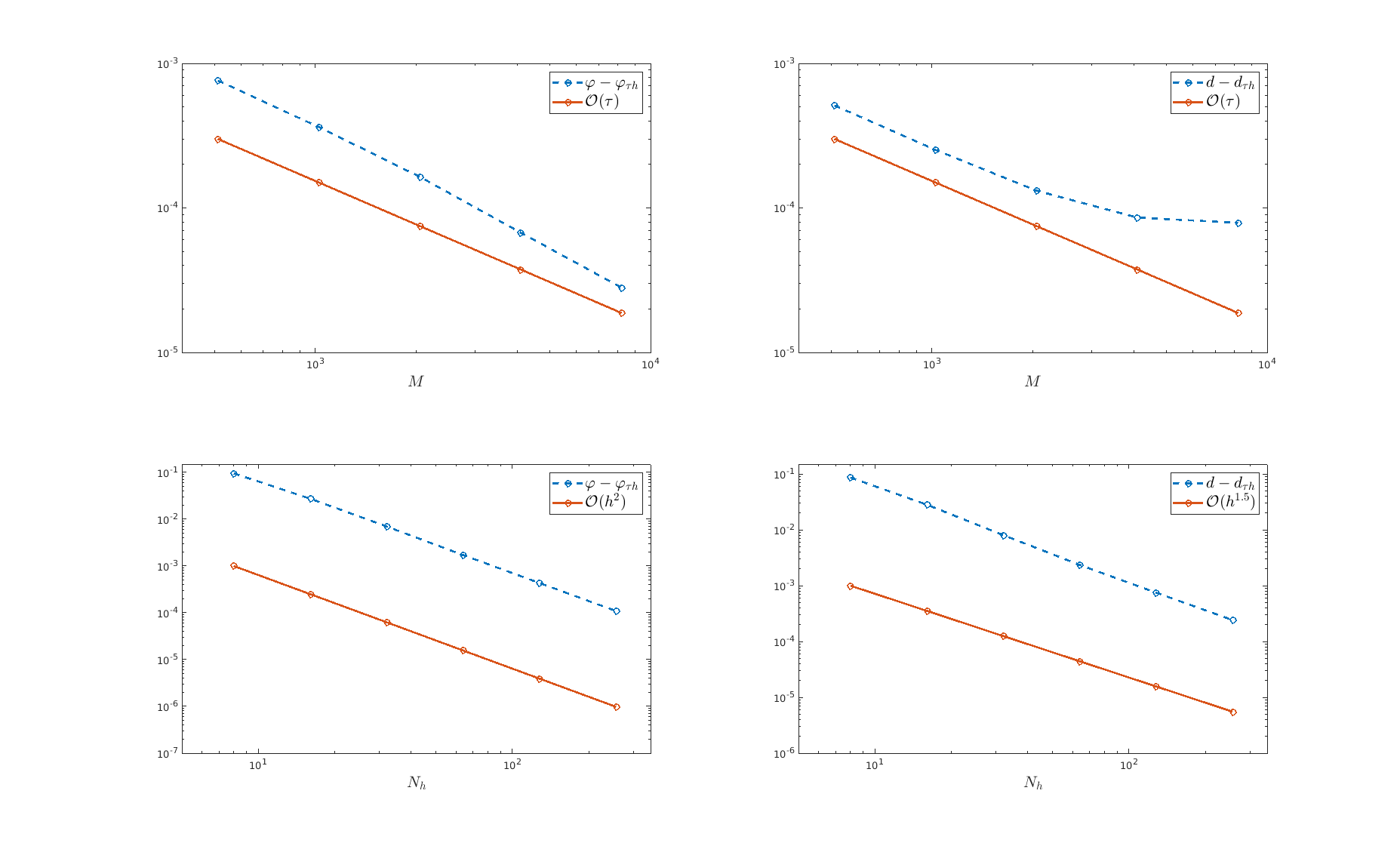} \caption{1st example: left: $\varphi_1$, right: $d_1$, top: temporal error, bottom: spatial error } 
\end{figure} 


Next, we have a look at a second example. This time, we set $\alpha=1,\beta=1, \delta=0.1, r=0.25\beta$ and choose
\[\varphi_2(t,x)=\begin{cases}9\frac{r}{\beta}(-27x^4+30x^3-12x^2+2x) &,\text{ if } x\in\left[0,\frac{1}{3}\right] \\ \frac{r}{\beta} &,\text{ if } x\in\left(\frac{1}{3},\frac{2}{3}\right)\\ 9\frac{r}{\beta}(-27x^4+78x^3-84x^2+40x-7) &,\text{ if } x\in \left[\frac{2}{3},1\right] \end{cases}\]
and 
\[d_2(t,x,y)=\begin{cases} 0 &, \varphi_2(t,x,y)\leq \frac{r}{\beta}, \\ (\varphi_2(t,x,y)-\frac{r}{\beta})(1-\exp(-\frac{\beta}{\delta}t)) &, \varphi_2(t,x,y)\geq \frac{r}{\beta}. \end{cases}\]
The corresponding load is given as 
\[l_2(t,x)=-\alpha\varphi^{\prime\prime}_2(t,x)+\beta\varphi_2(t,x)-\beta d_2(t,x).\]
The special feature of this example is, that there is a set of biactive points of positive measure present for all $t\in[0,1]$ since $-\beta(d_2-\varphi_2)-r=0$ for $x\in \left(\frac{1}{3},\frac{2}{3}\right)$. \\
{\small
\begin{table}[h] \centering
\begin{tabular}{|c|c c|c c||c|c c|c c|}
\hline
$h$&\multicolumn{4}{c||}{$2^{-11}$}&$\tau$&\multicolumn{4}{c|}{$2^{-9}$} \\
\hline
$\tau$& $\Vert e^\varphi_{\tau h}\Vert_{\ito}$ & EOC & $\Vert e^d_{\tau h}\Vert_{\ito}$ & EOC & $h$ & $\Vert e^\varphi_{\tau h}\Vert_{\ito}$ & EOC & $\Vert e^d_{\tau h}\Vert_{\ito}$ & EOC \\
\hline
$2^{-5}$ & 5.18e-06 & -    &  1.78e-04& - & $2^{-3}$ & 5.72e-02 & -    & 5.19e-02 & - \\
$2^{-6}$ & 2.55e-06 & 1.02 & 9.24e-05 & 0.94& $2^{-4}$ & 1.26e-02 & 2.17 & 1.11e-02 & 2.21\\
$2^{-7}$ & 1.20e-06 & 1.08 &  4.70e-05& 0.97 &$2^{-5}$ & 3.39e-03 & 1.90 & 3.01e-03 & 1.89\\
$2^{-8}$ & 7.38e-07 & 0.71 & 2.28e-05 & 1.04&$2^{-6}$ & 8.23e-04 & 2.04 & 7.09e-04 & 2.08 \\
$2^{-9}$ & 6.79e-07 & 0.12 & 1.05e-05 & 1.11&$2^{-7}$ & 2.08e-04 & 1.98 & 1.84e-04 & 1.94  \\
$2^{-10}$ & 7.26e-07 & -   & 4.60e-06 & 1.20& $2^{-8}$ & 5.16e-05 & 2.01 & 4.94e-05 & 1.90\\

\hline
\end{tabular}
\caption{2nd example: Errors for the states $\varphi_2, d_2$}
\label{tab:temporalerrorthirdexample}
\end{table}}


We require a very fine grid in space to ensure that the temporal error is dominant. Moreover, we observe second order convergence in space for both states. Both examples indicate that the proven rate of convergence for the error $\vt-\vth$ might not be optimal. If an error estimate of higher order for $\dt-\dth$ in $\LzweiHminuseins$ was available one could improve the results for $\vt-\vth$ in $\Lzweizwei$.
\begin{center}
\begin{figure} [h]
{\includegraphics[width=\textwidth]{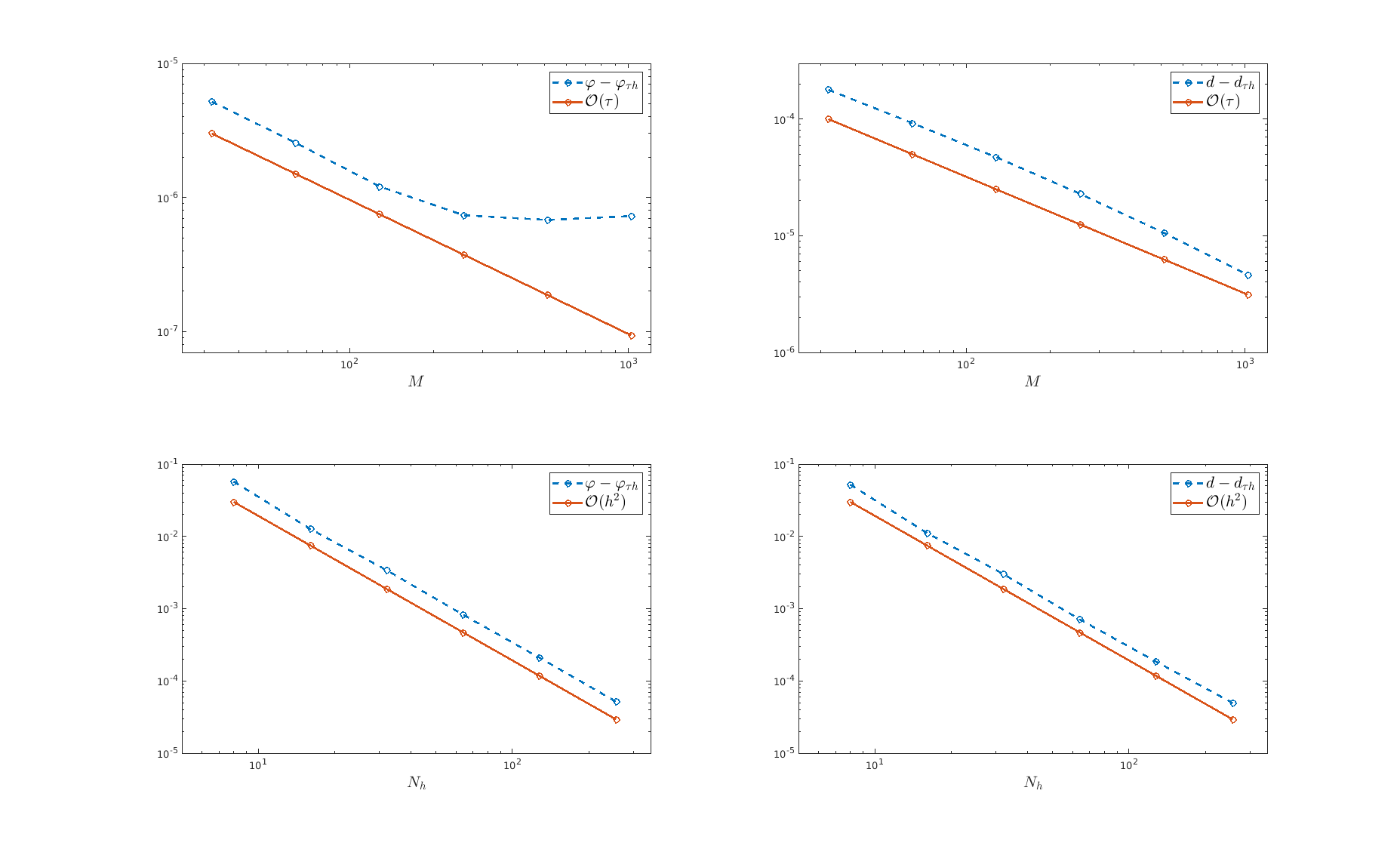} \caption{2nd example: left: $\varphi_2$, right: $d_2$, top: temporal error, bottom: spatial error }} 

\end{figure} 
\end{center}

 
\subsection{Optimization}
We will illustrate the experimental order of convergence for the $dG(0)cG(1)$-discretization of the control. Rather than solving the non-smooth optimal control problem with a suitable non-smooth algorithm, we employ a simple gradient's descent method combined with Armijo's line search for a regularized (and thus differentiable) optimal control problem $(P_\varepsilon)$. For the regularization of the $\max$-function, we make use of \eqref{ex:maxe} and choose $\varepsilon=10^{-9}$. We set $\alpha_l=10$ in both examples.\\
For the optimal control of the first example, we choose as objective
\[J(\varphi,d,l)=\frac{1}{2}\Vert \varphi-\varphi_1\Vert_{\ito}^2+\frac{1}{2}\Vert d-d_1\Vert_{\ito}^2+\frac{\alpha_l}{2}\Vert l-l_1\Vert_{\Heinsnullzwei}^2.\]
For the optimal control of the second example, we choose the full norm in the objective
\[J(\varphi,d,l)=\frac{1}{2}\Vert \varphi-\varphi_2\Vert_{\ito}^2+\frac{1}{2}\Vert d-d_2\Vert_{\ito}^2+\frac{\alpha_l}{2}\Vert l-l_2\Vert_{\Heinszwei}^2\]
since $l_2(0,\cdot)\not =0$. All results from the previous section are also fulfilled for the full norm. The optimal solution is given as $\ovl=l_i$ with corresponding states $(\ovv,\ovd)=(\varphi_i,d_i)=S(l_i)$ and objective value $j(\ovl)=0$ in both examples $i=1,2$. The optimality conditions of Theorem \ref{thm:optimalitysystemforP} are fulfilled with $(z,p)=(0,0)$ and $\mu=0$. Thus, the adjoint states and the multiplier exhibit high spatial regularity. Assumption \ref{ass:quadgrowthcond} is fulfilled due to Poincar\'{e}'s inequality in abstract function spaces. We abbreviate $e^{\ovl}_\sigma:= \ovl-\ovls$. Table \ref{tab:optimalcontrol} depicts the experimental order of convergence for the controls for different temporal and spatial meshes. 



{\small
 \begin{table}[h] \centering
\begin{tabular}{|c|c c|c|c c||c|c c|c|c c|}
\cline{1-6} \cline{7-12}
$h$&\multicolumn{2}{c|}{$2^{-9}$} & $\tau$ & \multicolumn{2}{c||}{$2^{-9}$} & $h$&\multicolumn{2}{c|}{$2^{-13}$} & $\tau$ & \multicolumn{2}{c|}{$2^{-9}$}  \\
\cline{1-6} \cline{7-12}
$\tau$& $\Vert e^{\ovl}_\sigma\Vert_{I\times\Omega}$ & EOC & $h$& $\Vert e^{\ovl}_\sigma\Vert_{I\times\Omega}$ & EOC & $\tau$& $\Vert e^{\ovl}_\sigma\Vert_{I\times\Omega}$ & EOC & $h$& $\Vert e^{\ovl}_\sigma\Vert_{I\times\Omega}$ & EOC  \\
\cline{1-6} \cline{7-12}
$2^{-7}$ & \multicolumn{2}{c|}{fixed point it.}      &$2^{-3}$ &6.89e-00 & - &  $2^{-5}$ &2.37e-04  & -     &$2^{-3}$ &1.50e-00 & -    \\
$2^{-8}$&  \multicolumn{2}{c|}{not conv.}    &$2^{-4}$ & 1.87e-00 & 1.87 &$2^{-6}$ &1.21e-04  &0.96   &$2^{-4}$ & 4.59e-01 & 1.70   \\
$2^{-9}$ &8.92e-02  &-   &$2^{-5}$& 4.78e-01 & 1.97 &  $2^{-7}$ &6.29e-05  &0.95   &$2^{-5}$& 1.24e-01 & 1.88 \\
$2^{-10}$ &4.21e-02  &1.08   &$2^{-6}$&1.43e-01 & 1.73 &$2^{-8}$ &3.59e-05  &0.80   &$2^{-6}$&4.06e-02 & 1.61  \\
$2^{-11}$ &1.87e-02  &1.16   &$2^{-7}$& 4.35e-02 & 1.72 & $2^{-9}$ &2.57e-05  &0.48   &$2^{-7}$& 1.27e-02 & 1.67 \\ 
$2^{-12}$&7.88e-03  &1.25   &$2^{-8}$& 1.39e-02 & 1.64 & $2^{-10}$&2.30e-05  &0.16   &$2^{-8}$& 4.36e-03 & 1.54 \\ 
\cline{1-6} \cline{7-12}
\end{tabular}
\caption{Errors for the controls, left: $l_1$, right: $l_2$}
\label{tab:optimalcontrol}
\end{table}}

\begin{center}
\begin{figure} [h]
{\includegraphics[width=\textwidth]{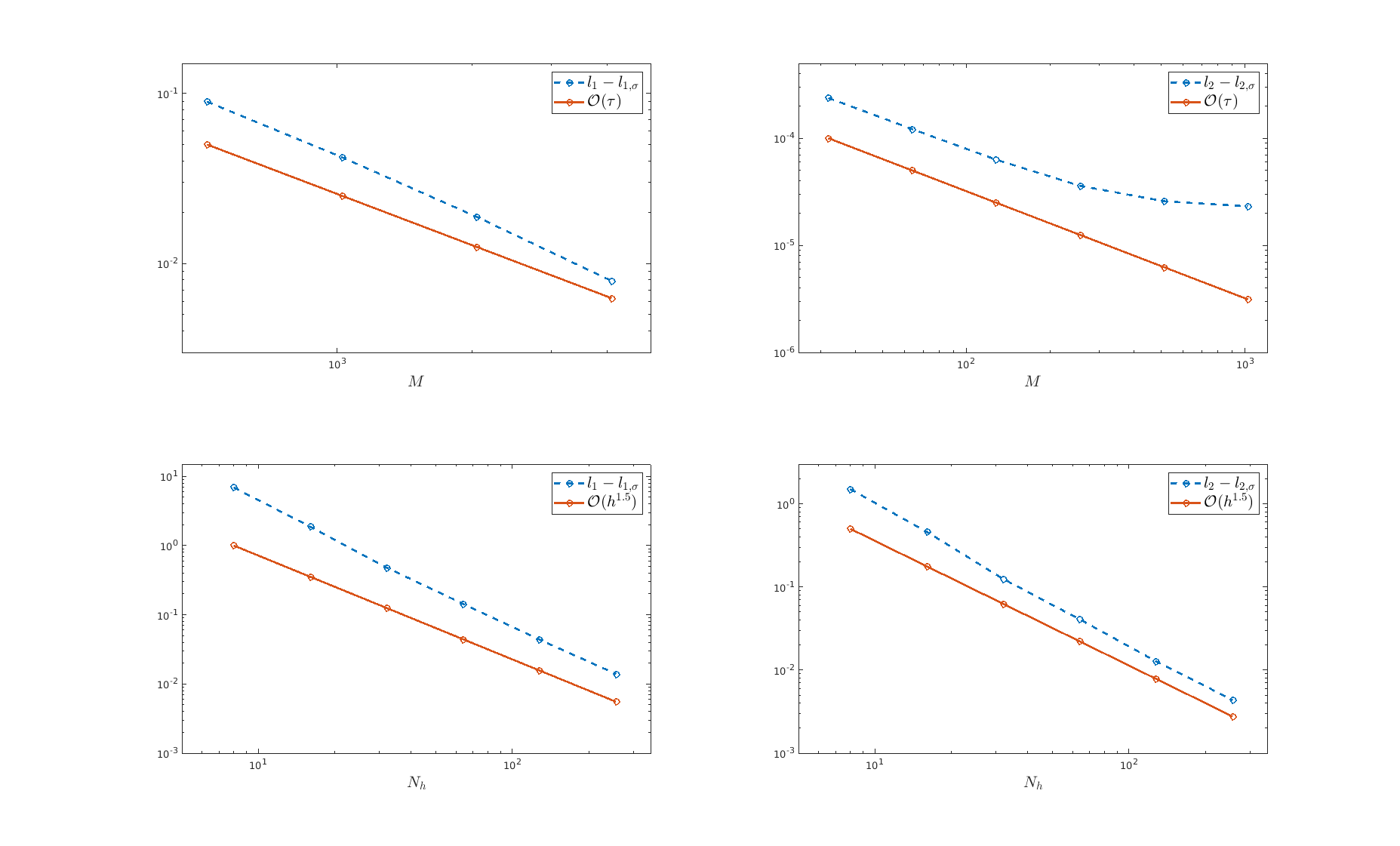} \caption{Errors for the optimal controls, left: $l_1$, right: $l_2$, top: temporal error, bottom: spatial error }} 

\end{figure} 
\end{center}
\vspace*{-0.5cm}
Both examples illustrate that the experimental rates for the control, both in time and space, are better than predicted by the theory. Whereas in time, we encounter the full convergence rate of $1$, the rates in space are close to $1.5$. These rates are similar to results for a $dG(0)cG(1)$ discretization of smooth problems, see for example \cite{NV11}.
Thus, a more detailed error analysis, based on  second order sufficient conditions for non-smooth problems,  might be promising. 

\bibliographystyle{abbrv}
\bibliography{Literatur}
\end{document}